\documentclass[11pt]{amsart}
\usepackage{latexsym,amssymb,amsmath}
\newtheoremstyle{custom}
  {3pt}
  {3pt}
  {\slshape}
  {}
  {\bfseries}
  {.}
  { }
   {}
\theoremstyle{custom}
\newtheorem{theorem}{Theorem}[section]

\newtheorem{proposition}[theorem]{Proposition}
\newtheorem{proposition/definition}[theorem]{Proposition/Definition}
\newtheorem{lemma}[theorem]{Lemma}
\newtheorem{corollary}[theorem]{Corollary}

\theoremstyle{definition}

\newtheorem{example}[theorem]{Example}

\theoremstyle{remark}
\newtheorem{remark}[theorem]{Remark}

\newtheoremstyle{exercise}
  {3pt}
  {6pt}
  {}
  {}
  {\bfseries}
  {:}
  { }
   {}
\theoremstyle{exercise}
\newtheorem{exercise}[theorem]{Exercise}
\newtheoremstyle{exercises}
  {3pt}
  {6pt}
  {}
  {}
  {\bfseries}
  {:}
  {\newline}
   {}

\theoremstyle{exercise}
\newtheorem{exercises}[theorem]{Exercises}

\input epsf
\def\boxit#1{\vbox{\hrule height1pt\hbox{\vrule width1pt\kern3pt
  \vbox{\kern3pt#1\kern3pt}\kern3pt\vrule width1pt}\hrule height1pt}}

\def\BC{\mathbb C}\def\BF{\mathbb F}\def\BO{\mathbb O}\def\BS{\mathbb S}

\def\BP{\mathbb P}
\def\pp#1{\mathbb P^{#1}}

\def\pp#1{{\mathbb P}^{#1}}
\def\tdim{{\rm dim}}

\def\hd{,...,}
\def\ww{\wedge}

\def\cE{{\mathcal E}}

\def\cS{{\mathcal S}}

\def\CC{\mathbb C}

\def\11{\mathbf 1}
\def\PP{\mathbb P}

\def\FF{\mathbb F}

\def\fg{{\mathfrak g}}

\def\l{\lambda}
\def\a{\alpha}

\def\o{\omega}

\def\b{\beta}

\def\s{\sigma}

\def\ot{{\mathord{ \otimes } }}
\def\op{{\mathord{\,\oplus }\,}}
\def\otc{{\mathord{\otimes\cdots\otimes}\;}}

\def\ra{{\mathord{\;\rightarrow\;}}}

\def\dim{{\rm dim}\;}

\def\op{\oplus}
\def\BO{\mathbb{O}}\def\BF{\mathbb{F}}

\def\op{\oplus}

\def\s{\sigma}
\def\t{\tau}

\def\a{\alpha}
\def\b{\beta}

\def\l{\lambda}

\def\ol{\overline}

\def\BP{\mathbb  P}
\def\BC{\mathbb  C}

\def\pp#1{\mathbb  P^{#1}}

\def\tcodim{\text{codim}}
\def\BS{\mathbb  S}

\def\fg{\mathfrak  g}

\def\hd{, \hdots ,}

\def\pp#1{\mathbb  P^{#1}}
\def\brank{\underline {\mathbf{ R}}}
\def\ur{\brank}

\def\ra{\rightarrow}

\def\tdet{\operatorname{det}}

\def\tim{\operatorname{Image}}
\def\tdim{\operatorname{dim}}

\def\tlim{\lim}
\def\tmod{\operatorname{mod}}

\def\ud#1{\underline d^{#1}}

\def\ww{\wedge}

\def\ctimes{\times \cdots\times}

\def\be{\begin{equation}}
\def\ene{\end{equation}}
\def\aaa{{\mathbf{a}}}
\def\bbb{{\mathbf{b}}}
\def\ccc{{\mathbf{c}}}

\newcommand{\Id}{\operatorname{Id}}

\newcommand{\Spec}{\operatorname{Spec}}

\numberwithin{equation}{section}
\numberwithin{theorem}{section}

\usepackage{color}
\newenvironment{red}{\color{red}}{}
\newcommand{\bred}{\begin{red}}
\newcommand{\ered}{\end{red}}
\newenvironment{blue}{\color{blue}}{}
\newcommand{\bblue}{\begin{blue}}
\newcommand{\eblue}{\end{blue}}
\newenvironment{green}{\color{green}}{}
\newcommand{\bgreen}{\begin{green}}
\newcommand{\egreen}{\end{green}}

\def\set#1{\left\{#1\right\}}
\def\fromto#1#2{#1, \dotsc, #2}
\def\setfromto#1#2{\set{\fromto{#1}{#2}}}

\def\ud{\mathrm{d}}
\newcommand{\JaBu}{Jaros\L{}aw Buczy\'n{}ski}
\newcommand{\shortJaBu}{J.~Buczy\'n{}ski}

\DeclareMathOperator{\Pf}{Pf}

\DeclareMathOperator*{\Times}{\times}

\newcommand{\Wedge}[1]{{\textstyle{\bigwedge\nolimits}^{\! #1}}}

\newcommand{\basept}{o}
\newcommand{\baseptaff}{\bar{\basept}}
\newcommand{\baseptline}{\hat{\basept}}

\newcommand{\prj}[1]{{\tilde #1}{}}
\newcommand{\inV}[1]{{#1}}
\newcommand{\nrml}[1]{{#1_N}}

\newcommand{\ccI}{{\mathcal{I}}}
\newcommand{\ccJ}{{\mathcal{J}}}
\newcommand{\ccT}{{\mathcal{T}}}

\renewcommand{\theenumi}{(\roman{enumi})}
\renewcommand{\labelenumi}{\theenumi}

\newcommand{\ptstypeii}{J(X,\t(X))}
\newcommand{\ptstypeiii}{\operatorname{Osc}(X)}
\newcommand{\ptstypeiv}{\operatorname{Z}(X)}

\title{On the third secant variety}
\author[\shortJaBu{} \& J.M.~Landsberg]{\JaBu \and  J.M. Landsberg}

\begin{document}

\begin{abstract}
We determine normal forms and ranks of tensors
of border rank at most three. We present a differential-geometric
analysis of limits of secant planes in a more general context. In particular
there are at most four types of points on limiting trisecant planes for cominuscule varieties such
as Grassmannians. We also show the singular locus of the secant varietites $\s_r(Seg(\pp n\times \pp m\times \pp q))$
has codimension at least two for $r=2,3$.
\end{abstract}
 \thanks{\shortJaBu{} supported by Marie Curie Outgoing Fellowship ``Contact Manifolds'',
         Landsberg supported by NSF grant   DMS-1006353}
\email{jabu@mimuw.edu.pl, jml@math.tamu.edu}
\address{Institut Fourier \\ Universit\'e Grenoble I \\
100 rue des Maths, BP 74\\
38402 St Martin d'H\`eres cedex, France\\
and
Institute of Mathematics of the
Polish Academy of Sciences\\
ul. \'Sniadeckich 8\\
P.O. Box 21\\
00-956 Warszawa, Poland}

\address{Department of Mathematics\\
Texas A\&M University\\
Mailstop 3368\\
College Station, TX 77843-3368, USA}

\maketitle

\section{Introduction}

Throughout the paper we work over the field of complex numbers $\BC$.

Motivated by applications, there has been a considerable amount of recent research on ranks
 and border ranks of tensors, see, e.g., \cite{Ltensorbook,MR2535056} and
references therein.
In signal processing one is interested in
determining ranks of tensors, see, e.g., \cite{Como02:oxford} and
references therein. In computational complexity, one looks
for exotic algorithms via limits of tensors of a given rank,
 see \cite{MR623057}. There are adequate tests
to determine the border ranks of tensors of small
border rank, however the possible ranks
of such tensors are not well understood.
In this article we present normal forms  for tensors of
border rank three. Already in this case the problem becomes subtle. We work in
the more general setting of secant varieties.

\subsection{Definitions, notational conventions} For a projective variety $X\subset \BP V$ not contained in
a hyperplane, the
\emph{$X$-rank} of $p\in \BP V$, $R_X(p)$,  is defined to be the smallest $r$ such that
there exist $x_1\hd x_r\in X$ such that $p$ is in the span of
of  $x_1\hd x_r$,
and the \emph{$X$-border rank} $\ur_X(p)$ is defined to be the smallest
$r$ such that there exist curves $x_1(t)\hd x_r(t)\in X$ such that $p$ is in the span of
 the limiting plane  $\tlim_{t\ra 0}\langle  x_1(t)\hd x_r(t)\rangle$.
Let $\s_r(X)\subset \BP V$ denote the set of points of
$X$-border rank at most $r$. When $X=Seg(\BP A_1\ctimes \BP A_n)
\subset \BP (A_1\otc A_n)$ is the set of rank one tensors in a space
of tensors, the $X$-rank and border rank agree with the usual
notions of tensor rank and border rank.
The set of points of $X$-rank   $r$
contains a Zariski open subset of $\s_r(X)$ and we are interested
in the complement of this set.

 We let  $\s_r^0(X)$ denote the points of $\s_r(X)$ of
rank $r$. The tangential variety of a smooth variety $X\subset \BP V$, $\t(X)\subset \BP V$, consists of all points
on all embedded tangent $\pp 1$'s. For varieties $X,Y\subset \BP V$, define
$$J(X,Y):=\ol{\{ p\in \BP V\mid \exists x\in X,\  y\in Y {\rm such\ that\ } p\in
\langle x,y\rangle\} },
$$
the {\it join} of $X$ and $Y$. Note that $J(X,X)=\s_2(X)$.
For a set $Z\subset \BP V$, $\hat Z\subset V$ denotes
the cone over it and $\langle Z\rangle$ its linear span. For a variety $Y\subset \BP V$, $Y_{sing}$ denotes the singular points of $Y$.
The affine tangent space to a variety $X\subset \BP V$ at a smooth point $x$ is denoted $\hat T_xX\subset V$.

Throughout the paper we assume $\fromto{A_1}{A_n}, A, B, C$ are complex vector spaces
    of dimension  at least $2$.

\subsection{Results on  ranks and normal forms
for tensors}

The following proposition was probably \lq\lq known to the experts\rq\rq\ but we did not find it
in the literature, so we include a statement and proof.

\renewcommand{\theenumi}{(\alph{enumi})}
\renewcommand{\labelenumi}{\theenumi}

\begin{proposition}\label{lastthm}
Let $X=Seg(\BP A_1\times \dots \times \BP A_n)\subset \BP (A_1\otc A_n)$ be a Segre variety.
There is a normal form for points  $x\in \hat\s_2(X)$:
\begin{enumerate}
 \item $x=a_1^1\otc a_1^n$ for a point of $X$, which has rank $1$,
 \item $x=a_1^1\otc a_1^n+ a_2^1\otc a_2^n$ for a point  on
       a secant line to $X$ (here  we   require at least two of  the $a^i_2$ to
       be independent of the  corresponding   $a^i_1$), which has rank $2$,
 \item  and  for
         each  $J\subseteq \setfromto{1}{n}$, $|J|>2$,    the normal form
       \be\label{tprimevect}
           x=\sum_{j \in J}a_1^1\otc a_{1}^{j -1}\ot a_2^{j }\ot a_1^{j +1}\otc a_1^n
       \ene
       where each
      $a^j_2$ is independent of  the corresponding  $a^j_1$. This case  has rank $|J|$.
\end{enumerate}
In particular, all ranks from $1$ to $n$ occur for elements of $\s_2(X)$.
\end{proposition}

Our main result is the analogous classification for points in the third secant variety of the Segre product:

\renewcommand{\theenumi}{(\roman{enumi})}
\renewcommand{\labelenumi}{\theenumi}

\begin{theorem}\label{s3nformthm}
Assume $n\geq 3$
 and let $X := Seg(\BP A_1\ctimes \BP A_n)$.
Let $p=[v]\in \s_3(X)\setminus \s_2(X)$.
Then $v$ has one of the following normal forms:
\begin{enumerate}
\item\label{item_main_thm_normal_form_honest_secant}
      $v=x+y+z$ with $[x], [y], [z] \in X$,
\item\label{item_main_thm_normal_form_point_plus_tangent}
      $v= x'+y$, with $[x],[y]\in X$ and $x'\in \hat T_{[x]}X$,
\item\label{item_main_thm_normal_form_third_order_pt}
      $v = x'+x''$, where $[x(t)]\subset X$ is a curve and $x'=x'(0)$, $x''=x''(0)$, or
\item\label{item_main_thm_normal_form_two_tangents}
      $v=x'+y'$, where $[x],[y]\in X$ are distinct points that
lie on a line contained in  $X$,
  $x'\in \hat T_{[x]}X$, and $y'\in \hat T_{[y]}X$.
\end{enumerate}
  The points of type \ref{item_main_thm_normal_form_honest_secant}
     contain a Zariski open subset of $\s_3(X)\setminus \s_2(X)$.
  If $\dim A_i \ge 3$, then
     those of type \ref{item_main_thm_normal_form_point_plus_tangent}
     have codimension one in $\s_3(X)$,
  those of type \ref{item_main_thm_normal_form_third_order_pt}
      are contained in the closure of those of type \ref{item_main_thm_normal_form_point_plus_tangent}
      and have codimension two in $\s_3(X)$,
  those of type \ref{item_main_thm_normal_form_two_tangents}
      are in the closure of the set of points of type \ref{item_main_thm_normal_form_third_order_pt}
      and have codimension four in $\s_3(X)$.
  There are  $n$ distinct components of  points of type \ref{item_main_thm_normal_form_two_tangents}.
  A general point of each type is not a point of any of the other types.
\end{theorem}

When $n=2$,  all points on $ \s_3(Seg(\BP A_1\times \BP A_2))\backslash  \s_2(Seg(\BP A_1\times \BP A_2))$ are of type \ref{item_main_thm_normal_form_honest_secant}.

The following result  may also have been  \lq\lq known to the experts\rq\rq\ but we did not find it
in the literature either:

\begin{theorem}\label{txsmooth}
A general point of $\t(Seg(\BP A\times \BP B\times \BP C))$, i.e., a
point with the normal form \eqref{tprimevect} with $|J|=3$, is a smooth point
of $\s_2(Seg(\BP A\times \BP B\times \BP C))$.
In particular
\[
\tcodim(\s_2(Seg(\BP A\times \BP B\times \BP C))_{sing},
\s_2(Seg(\BP A\times \BP B\times \BP C)))\geq 2.
\]
\end{theorem}

We prove an analogous result for $\s_3(Seg(\BP A\times \BP B\times \BP C))$:

\begin{theorem}\label{s3sing}
Let $p \in \s_3(Seg(\BP A\times \BP B\times \BP C))$.
If $p$ is a general point of type \ref{item_main_thm_normal_form_point_plus_tangent}
         or \ref{item_main_thm_normal_form_third_order_pt},
  or a general point of any component of points of type \ref{item_main_thm_normal_form_two_tangents},
  then $p$ is a nonsingular point of $\s_3(Seg(\BP A\times \BP B\times \BP C))$.
Moreover, if $\dim A, \dim B,\dim C \ge 3$,
   and $p$ is a general point in the set of the points contained in some $\PP(\CC^2 \otimes \CC^3 \otimes \CC^3)$,
   then $p$ is a nonsingular point of $\s_3(Seg(\BP A\times \BP B\times \BP C))$,
     and similarly for permuted statements.

In particular $\tcodim(\s_3(Seg(\BP A\times \BP B\times \BP C))_{sing},
\s_3(Seg(\BP A\times \BP B\times \BP C)))\geq 2$.
\end{theorem}

Normal forms for Theorem \ref{s3nformthm} when $n=3$ are as follows:
\begin{enumerate}
\item $a_1\ot b_1\ot c_1+ a_2\ot b_2\ot c_2+ a_3\ot b_3\ot c_3$
\item $a_1\ot b_1\ot c_2+ a_1\ot b_2\ot c_1+ a_2\ot b_1\ot c_1+ a_3\ot b_3\ot c_3$
\item $a_1\ot b_2\ot c_2+ a_2\ot b_1\ot c_2+a_2\ot b_2\ot c_1+a_1\ot b_1\ot c_3 +a_1\ot b_3\ot c_1 +a_3\ot b_1\ot c_1$
\item $ a_2\ot b_1\ot c_2 + a_2\ot b_2\ot c_1+a_1\ot b_1\ot c_3+a_1\ot b_3\ot c_1+a_3\ot b_1\ot c_1$.
\end{enumerate}
For type \ref{item_main_thm_normal_form_two_tangents} there are two other normal forms,
   where the role of $a$ is switched with that of $b$ and $c$.
These normal forms are depicted in terms of \lq\lq slices\rq\rq\  in
  Table~\ref{table_orbits_general} on page~\pageref{table_orbits_general}.
(In the tensor literature, $3$-way tensors $T\in A\ot B\ot C$ are often studied by their
images $T(A^*)\subset B\ot C$ etc... and these images are studied in terms of bases, resulting
in a parametrized subspace of a space of matrices.
These parametrized spaces of matrices are called {\it slices}.)
Here $a_j,b_j,c_j$ need not be independent vectors, so
to parametrize the spaces, fix bases of each space and
write the $a_j,b_j,c_j$ as arbitrary linear combinations
of basis vectors. (However there are some independence requirements.)

Here are normal forms for all $n$:
\begin{align}
\label{nf1}& p_{(i)}= a^1_{1}\otc a^n_{1}+ a^1_{2}\otc a^n_{2}+
a^1_{3}\otc a^n_{3} \\
&\label{nf2} p_{(ii)}= \sum_{i} a^1_{1}\otc a^{i-1}_{1}\ot a^i_2\ot a^{i+1}_{1}\otc
                    a^n_{1}
              + a^1_{3}\otc a^n_{3}\\
\label{nf3}
&
p_{(iii)}=\sum_{  i<j } a^1_{1}\otc a^{i-1}_{1}\ot a^i_2
\ot a^{i+1}_1 \otc
a^{j-1}_{1}\ot a^j_2
\ot a^{j+1}_1
\otc
 a^n_{1}
 \\ &\nonumber
+
\sum_{ i } a^1_{1}\otc a^{i-1}_{1}\ot a^i_3
\ot a^{i+1}_1 \otc
 a^n_{1}
\\
\label{nf4}
&p_{(iv)}=
\sum_{ s=2 }^{n} a^1_{2}\ot a^2_{1}\otc a^{s-1}_{1}\ot a^s_2
\ot a^{s+1}_1 \otc a^n_{1}.
\\
&\nonumber +
\sum_{ i=1 }^n a^1_{1}\otc a^{i-1}_{1}\ot a^i_3
\ot a^{i+1}_1 \otc
 a^n_{1}
\end{align}
Again, \eqref{nf4} has $n-1$ other normal forms, where the role of $a^1_{*}$ is exchanged with $a^i_{*}$.
Also, the vectors need not all be linearly independent.

\begin{remark}
In contrast to  case (iv)  above, already with four points on a three factor Segre spanning a three dimensional
vector space, one can obtain new limits by taking a second derivative, even when the limiting
points are distinct.
Consider the points $x_1=a_1\ot b_1\ot c_1$, $x_2=a_2\ot b_2\ot c_1$,
$x_3=\frac 12(a_1+a_2)\ot( b_1-b_2)\ot c_1$, $x_4=\frac 12(a_1-a_2)\ot( b_1+b_2)\ot c_1$.
Note that $x_1=x_2+x_3+x_4$. Here both first and second derivatives of curves give
new points.
More generally,  consider
\[
  Seg(v_2(\pp 1)\times \underbrace{\pp 0\ctimes \pp 0}_{(n-2)\text{ factors}}) \subset Seg(\BP A_1\otc \BP A_n).
\]
 Any four
points lying on
 $
    Seg(v_2(\pp 1)\times \pp 0\ctimes \pp 0)
 $
  will be linearly dependent.
Exceptional limit points turn out to be important - an exceptional limit in
$\s_5(Seg(\BP A\times \BP B\times \BP C))$ is used in Bini's
approximate algorithm to multiply $2\times 2$ matrices with an entry  zero,
and an exceptional limit in  $ \s_7(Seg(\BP A\times \BP B\times \BP C)) $ is used in Sch\"onhage's
approximate algorithm to multiply $3\times 3$ matrices using $21$ multiplications,
see \cite[\S4.4]{jabu_ginensky_landsberg_Eisenbuds_conjecture}.
\end{remark}

\renewcommand{\theenumi}{(\arabic{enumi})}
\renewcommand{\labelenumi}{\theenumi}

Since there are only finitely many configurations of triples of  points in $A_i$ up to the action of $GL(A_i)$, we   conclude:
\begin{corollary}
   There are only finitely many orbits of the action of
       $GL(A_1) \times \dots \times GL (A_n)$ on $\sigma_3 (Seg(\BP A_1\times \dots \times \BP A_n))$.
\end{corollary}
In the three factor case, there are $39$ orbits, see  \S\ref{sec_ranks_and_orbits}.

\begin{remark}
Points of the form $y+y'+y''$
   where $y(t)$ is a curve on $\hat Seg(\BP A_1\otc \BP A_n)$ have
rank at most $\binom {n+1}2$ because
all such points are of the form \eqref{nf3} (perhaps with linearly dependent variables).
The bound $R_{Seg(\BP A_1\otc \BP A_n)}(y+y'+y'') \le \binom{n+1}{2}$ is not  tight,
  as for $n=3$ the following theorem
shows $R_{Seg(\BP A\times \BP B \times \BP C)}(y+y'+y'')$ is at most five.
\end{remark}

\begin{theorem}\label{lastcor} The rank of a general point of the form $[y+y'+y'']$ of
$ \s_3(Seg(\BP A\times \BP B\times \BP C))$  as well as the
rank of a general point of the form $[x'+y']$ where $[x],[y]$ lie on a line in $Seg(\BP A\times \BP B\times \BP C)$,
 is $5$.
All other points of $ \s_3(Seg(\BP A\times \BP B\times \BP C))$
have rank less than five, so in particular,   the maximum rank of any point of $\s_3(Seg(\BP A\times \BP B\times \BP C))$ is $5$.
\end{theorem}

\begin{remark} Theorem \ref{lastcor} seems to have been a \lq\lq folklore\rq\rq\ theorem
in the tensor literature. For example, in \cite{MR2535056}, Table 3.2 the result
is stated and refers to \cite{MR1088949}, but in that paper the result is stated
and a paper that never appeared is referred to. Also, there appear to have been privately circulating
proofs, one due to R. Rocci from 1993 has been shown to us.
We thank M. Mohlenkamp for these historical remarks.
\end{remark}

The Comon conjecture on ranks says that for
$T\in S^dV\subset V^{\ot d}$
the symmetric tensor rank of $T$ equals the tensor rank of $T$.

\begin{corollary}\label{comoncor} The Comon conjecture holds for
$T\in \hat \s_3(v_3(\BP V))$.
\end{corollary}

Corollary \ref{comoncor} follows by comparing the normal forms and ranks of this paper with those
of \cite{LTrank}.
\smallskip

In \S\ref{gencomin} we generalize Theorem \ref{s3nformthm} to generalized cominuscule varieties,
a class of homogeneous varieties which includes Grassmannians and spinor varieties.
See \S\ref{gencomin} for the definition of a generalized cominuscule variety,  and \S 2 for the definition of the second fundamental form $II$.

\begin{theorem}\label{thm_points_in_Sigma} Let $X\subset \BP V$ be generalized cominuscule.
   Then $[\inV{p}] \in \sigma_3(X)$
   if and only if
   at least
   one of the following situations occurs:
   \begin{itemize}
    \item[\ref{item_main_thm_normal_form_honest_secant}] 
           $\inV{p} = \inV{\xi} + \inV{\eta} +\inV{\zeta}$
           for some linearly independent $\inV{\xi}, \inV{\eta}, \inV{\zeta} \in \hat X$ ($\inV{p}$ is on an honest $3$-secant plane),
    \item[\ref{item_main_thm_normal_form_point_plus_tangent}] 
           $\inV{p} = \inV{\xi'} +\inV{\eta}$ for some $\inV{\xi}, \inV{\eta} \in \hat X^0$
           and $\inV{\xi'} \in \hat{T}_{[\inV{\xi}]} X$,
    \item[\ref{item_main_thm_normal_form_third_order_pt}]
           $\inV{p} = \inV{\xi'} + II( (\eta')^2)$ for some $\inV{\xi} \in \hat X$,
           $\inV{\xi'} \in  {\hat T}_{[\inV{\xi}]} X$,  $\inV{\eta'} \in  {T}_{[\inV{\xi}]} X$, or
    \item[\ref{item_main_thm_normal_form_two_tangents}]  
           $\inV{p} = \inV{\xi'} + \inV{\eta'}$ for some $\inV{\xi},\inV{\eta}  \in \hat X$,
           $\inV{\xi'}  \in  { \hat T}_{[\inV{\xi}]} X$,
           $\inV{\eta'}  \in  {\hat  T}_{[\inV{\eta}]} X$ with the line $\BP \langle \xi,\eta\rangle$ contained in $X$.
   \end{itemize}
To make sense of elements of the tangent and normal spaces as elements of $V$ we
have chosen a splitting  $V=\hat x\op T\op N$  as described in
   \S\ref{taylorsect}.
\end{theorem}

\subsection{Overview}
In \S\ref{diffgsect} we review facts from projective differential geometry.
In \S\ref{gencomin} we prove Theorem \ref{thm_points_in_Sigma}.
In \S\ref{s3exsect} we apply Theorem \ref{thm_points_in_Sigma} to cominuscule varieties,
   including     Grassmannians  and   spinor varieties.
In  \S\ref{smallsecsegsect} we analyze the case of the Segre variety  in detail,
  and we give two  proofs of Theorem~\ref{s3nformthm}, a short proof by computing the
Lie algebras of the stabilizers of the points $p_{(*)}$, and a longer   proof that contains
more precise information which is of  interest in its own right.
In \S\ref{sec_ranks_and_orbits} we restrict attention to the three-factor Segre variety,
   and   prove   Theorems~\ref{txsmooth}, \ref{s3sing} and \ref{lastcor}.

\subsection{Acknowledgments}
We thank   M. Mohlenkamp for   pointing out
an error  in an  earlier version of this article,
 related to  the rank of
$y+y'+y''$ in Theorem~\ref{lastcor}.
This paper grew out of questions raised at the
2008 AIM workshop {\it Geometry and representation theory of tensors
for computer science, statistics and other areas}, and the authors
thank AIM and the conference participants for inspiration.
The mathematics in this paper was finally completed while the authors were
guests at the Mittag-Leffler Institut in Spring 2011 and we gratefully
thank the institute for providing a wonderful environment for
doing mathematics.
 We truly appreciate the help of the referee,
      his careful proof reading of the article,
      and his many thoughtful comments.

\section{Curves in submanifolds of projective space}\label{diffgsect}

\subsection{Fubini forms, fundamental forms, and the prolongation property}\label{taylorsect}
Let $X^n\subset  \BP V$ be a subvariety and let $\basept \in X$ be a smooth point.
We may choose a splitting
\be\label{splitting}
V = \baseptline \oplus T \oplus N,
\ene
  such that $\baseptline \simeq \CC$
  is the one dimensional linear subspace corresponding to $\basept \in \PP V$,
  and $\baseptline \oplus T$ is the affine tangent space $\hat T_{\basept} X$.

We will abuse notation and identify $T$ with  the Zariski tangent space
$T_\basept X= \hat o^*\ot (\hat T_{\basept}X/\hat o)$ and $N$ with the normal space
$N_\basept X:=T_{\basept}\BP V/T_{\basept}X$. Since we are working at a point, the twist by the line
bundle will not matter. Our choice of splitting will not effect the end
results of the calculations.

Any  point  $[\inV{v}] \in \PP V$ has a  lift to a point $\inV{v} \in V$
    of the form  $(\baseptaff, \prj{v}, \nrml{v})$ or $(0, \prj{v}, \nrml{v})$,
  where $0$ and $\baseptaff$ are points in $\baseptline \simeq \CC$,
  and $\prj{v} \in T$, $\nrml{v} \in N$.
In an analytic neighborhood of $\basept$ we may write $X$ as a graph, that is,  for $\inV{x} \in X$ near $\basept$,
the vector  $\nrml{x}$ depends holomorphically on the vector  $\prj{x}$ and we expand this holomorphic map into a Taylor series:
\begin{equation}\label{taylore}
\nrml{x}=\nrml{x}(\prj{x})=
II_{\basept}(\prj{x}^2) + F_{3, \basept}(\prj{x}^3) + F_{4, \basept}(\prj{x}^4)+\dotsb
\end{equation}
 Here $\prj{x}\in T$ and $\prj{x}^s\in S^sT$.
 Later we will study curves $x(t) \in X$, and express the whole curve using \eqref{taylore},
writing $\prj{x}(t)$ to be the curve in $T$,  $\prj{x}^s(t) \in S^sT$.
Note that by our choice of splitting there is no constant or linear term in  \eqref{taylore}.
The quadratic part $II_{\basept}= F_{2,\basept}$ gives rise to a well defined tensor in $S^2 T^*_{\basept}X\ot N_{\basept} X \simeq S^2 T^*\ot N$,
   called the \emph{second fundamental form}.
Further, $F_{s,{\basept}} \in S^sT^*\ot N$
  are called the {\it Fubini forms},
  but they depend on the choice of splitting $V = \baseptline \oplus T \oplus N$.
See \cite[Chap. 3]{IvL} for more details.

\smallskip

One can extract   tensors from the Fubini forms, called  \emph{fundamental forms}.
Let
$$N_{s,{\basept}}:=
N_{\basept}X\tmod \tim(F_{2,{\basept}}\hd F_{s-1,{\basept}}),
$$
the tensor   $\BF_{s,{\basept}}:= \bigl( F_{s,{\basept}} \tmod \tim(F_{2,{\basept}}\hd F_{s-1,{\basept}})\bigr)
             \in S^sT^*_{\basept}X\ot  N_{s,{\basept}}$
  is well-defined (independent of the choice of splitting \eqref{splitting}) and called the {\it $s$-th fundamental form} of $X$ at~${\basept}$.
Fundamental forms satisfy a {\it prolongation property} (see \cite[Chap. 3]{IvL}):
if ${\basept}\in X$ is a general point, then for all $f_1\in S^{s_1} T$ and $f_2 \in S^{s_2} T$ we have
\begin{equation}\label{equ_prolongation}
  \BF_{s_1,{\basept}}(f_1) = 0  \Longrightarrow \BF_{s_1+s_2,{\basept}}(f_1 f_2) =0.
\end{equation}
We write $III_{\basept}=\BF_{3,\basept}$.
If there is no risk of confusion, we will often omit the base point and write $II := II_{\basept}$, $F_s:= F_{s,\basept}$, etc.

\subsection{When taking limits, we may assume one curve is stationary}
\begin{lemma}\label{lem_one_curve_is_constant}
   Let $G$ be a connected algebraic group and $P$ a parabolic subgroup.
   Let $X=G/P\subset \BP V$ be a homogeneously embedded homogeneous
variety and let  $p \in \s_r(X)$. Then there exist a point $\inV{\xi} \in  \hat{X}$
       and $r-1$ curves $\inV{y_j}(t) \in \hat X $  such that
       $p \in  \lim_{t \to 0} \langle \inV{\xi}, \inV{y_1}(t)\hd  \inV{y_{r-1}}(t)\rangle$.
\end{lemma}
\begin{proof}
   Since $p \in \s_r(X)$, there exist $r$ curves $\inV{x}(t), \inV{y_1}(t)\hd  \inV{y_{r-1}}(t) \in \hat{X}$
      such that
$$\inV{p} \in \lim_{t \to 0} \langle \inV{x}(t), \inV{y_1}(t)\hd  \inV{y_{r-1}}(t)\rangle.$$
   Choose  a curve $g_t \in G$, such that $g_t(\inV{x}(t)) = \inV{x}_0 = \inV{x}(0)$ for all $t$  and $g_0=Id$.
  We have  \begin{align*}
     \langle \inV{x}(t), \inV{y_1}(t)\hd  \inV{y_{r-1}}(t)\rangle & =
         {g_t}^{-1} \cdot \langle \inV{x}_0, g_t \cdot\inV{y_1}(t)\hd  g_t \cdot\inV{y_{r-1}}(t)\rangle \text{ and}\\
     \lim_{t \to 0} \langle  \inV{x}(t), \inV{y_1}(t)\hd  \inV{y_{r-1}}(t)  \rangle & =
         \lim_{t \to 0} \bigl({g_t}^{-1} \cdot \langle \inV{x}_0, g_t \cdot\inV{y_1}(t)\hd  g_t \cdot\inV{y_{r-1}}(t)\rangle \bigr) \\
       &   {=}
          \lim_{t \to 0}  \langle \inV{x}_0, g_t \cdot\inV{y_1}(t)\hd  g_t \cdot \inV{y_{r-1}}(t) \rangle . \\
  \end{align*}
  Set  $\inV{\xi} = \inV{x_0}$   and appropriately modify the  $\inV{y_j}(t)$
to complete the proof.
\end{proof}

We remark, that for non-homogeneous $X$, an analogous statement is rarely true.
If $r=2$, and $X$ is smooth, then it is true, see Proposition~\ref{prop_case_r=2}.
But already if $r=2$ and $X$ is singular, one often needs both curves moving
   (a cuspidical rational curve embedded in $\PP^3$ is an example).
Also if $r=3$, and $X$ has a trisecant line
  (for example $X$ is a high degree rational normal curve projected from a general point on a trisecant plane),
  then one  also needs three curves moving to obtain some of the points on the third secant variety.

\subsection{Dimension counting and higher order invariants}
Since $\tdim \s_r(X)\leq r\tdim X+r-1$, one can use a parameter
count to see what one expects in choosing a point of the boundary.
Suppose $\tdim X>1$, $X$ is not a cone and
the third fundamental form
  is nonzero --- for example $X= Seg(\BP A\times \BP B\times \BP C)$.
One can predict that the third fundamental form does not arise when computing
   a point of $\s_3(X)$   which is  on a plane obtained as a limit of spans of $3$ points converging to the same  general point of $X$.
This is because the third fundamental form is only well defined modulo the second osculating space,
  which will have dimension greater than $\tdim X$.
In the case of the three factor Segre variety
  the second  osculating space has dimension $\aaa\bbb+\aaa\ccc+\bbb\ccc$,
  and the third fundamental form is only well defined modulo
  the second osculating space.
So were there a term  $III(v^3)$  appearing in an expression
  for a point on $\s_3(Seg(\BP A\times \BP B\times \BP C))$, with no restrictions
  on $v$, then the resulting variety would
  have to have dimension at least $\aaa\bbb+\aaa\ccc+\bbb\ccc$
  for the term to be well defined.
If the dimensions of the vector spaces are sufficiently large, this contradicts the dimension count.
Such heuristics can be useful in calculations.

The following lemma will allow us to  eliminate higher fundamental forms from our considerations
  when studying $\s_3(X)$.
It illustrates the dimension counting principle.

\begin{lemma}\label{lem_vanishing_of_II_gives_vanishing_of_F_s}
Let $X\subset \BP V$ be a variety and let $\basept\in X$ be a general point. Adopt the notations
of \S\ref{taylorsect}.
  Suppose $\prj{v}(t) \subset T$ is a curve such that $II(\prj{v}(t)^2)$ vanishes at $t=0$ up to order $m-1$,
    that is $II(\prj{v}(t)^2) = t^m (\cdots)$.
  If $m>0$ and $s\ge 2$, then $\FF_s(\prj{v}(t)^s)$ vanishes at $t=0$ up to order $m+s-3$,
 that is $\FF_s(\prj{v}(t)^s)= t^{m+s-2} (\cdots)$.
\end{lemma}
\begin{proof}
  Let $\ccI^d:=\set{f \in S^d T \mid \FF_d(f) =0}$.
  Since $\ccI^d$ is a linear subspace of $S^d T$,
     the prolongation property \eqref{equ_prolongation} implies  $\ccI^{d_1} \cdot S^{d_2} T \subset \ccI^{d_1+d_2}$.
  Thus, if $S:=\bigoplus_{d=0}^{\infty} S^d T$ is the symmetric algebra,
    and $\ccI := \bigoplus_{d=0}^{\infty} \ccI^d$, then $\ccI$ is a homogeneous ideal.

  Consider $S[[t]]$,  the   power series ring with coefficients in  $S$,   and let $\ccJ_k$
    be the ideal generated by $\ccI$ and $t^k$.
  The curve $\prj{v} = \prj{v}(t) = \prj{v}_0 + t \prj{v}_1 + t^2 \prj{v}_2 +\dotsb$
     is naturally an element in $S[[t]]$.
  In this interpretation,
    $\FF_s(\prj{v}(t)^s) = t^k   (\cdots)$  if and only if $\prj{v}(t)^s \in \ccJ_{k}$.
  In particular, our assumptions are:
  \begin{itemize}
     \item  $\prj{v}(t)^2 \in \ccJ_{m}$ and
     \item  the constant coefficient $\prj{v}_0^2 \in \ccI$ (because $m >0$), thus also $\prj{v}_0^s \in \ccI$ for $s \ge 2$.
  \end{itemize}
  To show  that $\prj{v}(t)^s \in \ccJ_{m+s-2}$ for $s\ge 2$,
  we argue by induction on $s$.
Consider $\frac{\partial}{\partial t} \left(\prj{v}(t)^s\right) = s \frac{\partial \prj{v}}{\partial t} \prj{v}^{s-1}$.
  By the inductive assumption,  $\prj{v}^{s-1}\in \ccJ_{m+s-3}$,
    so    $\frac{\partial}{\partial t} \left(\prj{v}(t)^s\right) \in \ccJ_{m+s-3}$.
  Since the constant coefficient $\prj{v_0}^s \in \ccI$, it follows,
    that $\prj{v}(t)^s \in \ccJ_{m+s-2}$ as claimed.
\end{proof}

\subsection{Points on $\s_2(X)$}

We reprove the standard fact that a point
on a secant variety  to a smooth variety $X$  is either  on $X$, on
an  honest secant line, or on a tangent line to $X$.
The proof we present prepares the way for new results.
Recall that if a point of $\s_2(X)$ is not on an
honest secant line, it must arise from a point on a limiting $\pp 1$ which is obtained
by a curve of $\pp 1$'s,  $\langle x(t), y(t)\rangle$ where $[x(0)]=[y(0)]$.

\begin{proposition}\label{prop_case_r=2}
Let $X\subset \BP V$ be a smooth variety and let $[z]\in \s_2(X)\backslash \s_2(X)^0$. Then $z$
may be obtained from first order information, that is,
 $z=u'$ for some  $[u]\in X$ and $u'\in \hat T_{[u]}X$.
\end{proposition}

\begin{proof} There exist curves
  $[\inV{x}(t)],[\inV{y}(t)]\subset X$   with
  $\inV{x}(0) =  \inV{y}(0) =\baseptline \in \basept \setminus \set{0}$,
  such that  $[z]$   may be obtained
  as a point of the limiting $\pp 1=\BP(\tlim_{t\ra 0} \langle \inV{x}(t),\inV{y}(t)\rangle)$.

   Consider a splitting $V = \hat \basept \oplus T \oplus N$
       and the   curves $\prj{x}(t), \prj{y}(t) \in T$ as above.
   Write:
   \begin{align*}
      \prj{x}(t) & =   \prj{x_1} t + \prj{x_2} t^2 + \dotsb + \prj{x_{k-1}} t^{k-1} + \prj{x_k} t^{k} + \prj{x_{k+1}} t^{k+1} + \dotsb\\
      \prj{y}(t) & =   \prj{x_1} t + \prj{x_2} t^2 + \dotsb + \prj{x_{k-1}} t^{k-1} + \prj{y_{k}} t^{k} + \prj{y_{k+1}} t^{k+1} + \dotsb
   \end{align*}
where $\prj{x_j},\prj{y_j}\in T$
   and   $k$ is the smallest integer such  that $\prj{v_0} := \prj{y_k}- \prj{x_k} \ne 0$.
Let $\prj{v}(t) := t^{-k}(\prj{y}(t) -\prj{x}(t)) = (\prj{y_k}-\prj{x_k}) + (\prj{y_{k+1}}-\prj{x_{k+1}}) t + \dotsc$.
   Then:
   \begin{align*}
      \inV{y}(t) &- \inV{x}(t)
           = (\baseptaff  + \prj{y}(t) + II(\prj{y}(t)^2) + F_3(\prj{y}(t)^3) + \dotsb)
            - (\baseptaff  + \prj{x}(t) + II(\prj{x}(t)^2) + F_3(\prj{x}(t)^3) + \dotsb) \\
          & = t^k \prj{v}(t) + II\left(\prj{y}(t)^2 - \prj{x}(t)^2\right) + F_3\left(\prj{y}(t)^3 - \prj{x}(t)^3\right) + \dotsb \\
          & = t^k \prj{v}(t) + II\left((\prj{y}(t) - \prj{x}(t))(\prj{x}(t)+\prj{y}(t))\right) + F_3\left((\prj{y}(t) - \prj{x}(t))(\prj{x}(t)^2 + \prj{x}(t)\prj{y}(t) + \prj{y}(t)^2)\right) + \dotsb \\
          & = t^k \prj{v}(t) + II\left(t^k \prj{v}(t) (\prj{x}(t)+\prj{y}(t))\right) + F_3\left(t^k \prj{v}(t)(\prj{x}(t)^2 + \prj{x}(t)\prj{y}(t) + \prj{y}(t)^2)\right) + \dotsb
   \end{align*}
   Since $\prj{x}(t)$ and $\prj{y}(t)$ have no constant terms, we obtain:
   \begin{align*}
      \inV{y}(t) - \inV{x}(t) & = t^k \prj{v_0} + t^{k+1} (\dotsb) \text{ and} \\
      \inV{x}(t) \wedge \inV{y}(t)
           & = \inV{x}(t) \wedge (\inV{y}(t) - \inV{x}(t)) \\
           & = \left( \baseptaff + t(\dotsc) \right) \wedge  \left( t^k \prj{v_0} + t^{k+1} (\dotsb) \right)\\
           & = t^k \left( \baseptaff  \wedge   \prj{v_0} \right) +  t^{k+1}(\dotsc).
   \end{align*}
   Recall that $\prj{v_0} \ww \baseptaff \ne 0$.
   Thus the limiting affine plane $\tlim_{t\ra 0} \langle \inV{x}(t),\inV{y}(t)\rangle)$ is equal to
        $\langle\baseptaff, \prj{v_0} \rangle$.

   Set $\prj{z}(t) := t\prj{v}(t) \in T$ and $\inV{z}(t) := \baseptaff + t \prj{v}(t) + II(t^2 \prj{v}(t)^2) + \dotsb \in \hat X$.
   Then the same affine plane can be obtained as
      $\tlim_{t\ra 0} \langle \baseptaff,\inV{z}(t)\rangle$,
      thus one point is fixed and the other approaches the first one from the direction of $\prj{v_0}$.
\end{proof}

\section{Generalized cominuscule varieties: proof of theorem \ref{thm_points_in_Sigma} }\label{gencomin}

Following \cite{LWtan}, a  homogeneously embedded homogeneous variety $G/P\subset \BP V$ is called \emph{generalized
cominuscule} if there is a choice of splitting (at any point) such that the Fubini forms
reduce to fundamental forms, that is:
\begin{equation}\label{equ_splitting_for_gen_cominuscule}
 V = \baseptline \oplus T \oplus N_2 \oplus N_3 \oplus \dotsb \oplus N_f
\end{equation}
  with $F_s(S^s T) \subset N_s$ and thus $F_s = \FF_s$ for all $s \in \setfromto{2}{f}$,
  and $F_s = \FF_s = 0$ for all $s > f$.
Generalized cominuscule varieties may be characterized intrinsically
as  the homogeneously embedded $G/P$ where the unipotent radical of $P$ is abelian.    A generalized cominuscule variety is cominuscule
 if and only if   $G$ is simple and the embedding is the minimal homogeneous one.
For those familiar with representation theory, a homogeneously embedded homogeneous variety $G/P\subset \BP V$
is cominuscule if $V$ is a fundamental representation $V_{\o_i}$ where $\o_i$ is a cominuscule weight, that is,
the highest root of $\fg$ has coefficient one on the simple root $\a_i$. Generalized cominuscule varieties
are Segre-Veronese embeddings of products of cominuscule varieties.

Grassmannians $G(k,W)$, projective spaces $\PP^n$ and
   products of projective spaces in any homogeneous embedding
   (in particular, respectively, $G(k,W)$ in the Pl\"ucker embedding, Veronese varieties,  and    Segre varieties)
    are generalized cominuscule.

Throughout this section we assume $X$ is generalized cominuscule.
When studying points of $\s_3(X)$, one has to take into account
curves limiting to points on a trisecant line of $X$. When $X$ is cut out
by quadrics, as with homogeneous varieties,
any trisecant line of $X$ will be contained in $X$.
Theorem~\ref{thm_points_in_Sigma}
shows such points are already accounted for
by curves with just one or two limit points, and that higher order differential invariants
do not appear,
as was hinted at in  Lemma \ref{lem_vanishing_of_II_gives_vanishing_of_F_s}.

We commence the proof of Theorem \ref{thm_points_in_Sigma}  with an observation about the freedom of choice of splitting as in \eqref{equ_splitting_for_gen_cominuscule}.

\begin{lemma}\label{lem_can_pick_splitting}
    Let $X$ be generalized cominuscule and let $x$, $\fromto{y_1}{y_{r-1}}$ be $r$ points on $X$.
    Then there exists a choice of splitting as in \eqref{equ_splitting_for_gen_cominuscule} (so $F_s(S^sT) \subset N_s$ for all $s$),
      such that $x = \basept$ is the center of this splitting and none of the points $\fromto{y_1}{y_{r-1}}$
      lies  on the hyperplane $T\oplus N_2 \oplus N_3\oplus \dotsb$.
\end{lemma}

\begin{proof}
    Let $G$ be the automorphism group of $X$ and $P \subset G$ be the parabolic subgroup preserving $x$.
    Let $Y \subset X\times \BP V^*$    be the set of those pairs $(\basept, H)$, where $\basept \in X$ and $H \subset V$
      is a hyperplane, such that $V = \baseptline \oplus H$
      and there exists a splitting $H= T \oplus N_2 \oplus N_3\oplus \dotsb$,
      making the splitting of $V$ as in \eqref{equ_splitting_for_gen_cominuscule}.
    Since $X$ is generalized cominuscule, $Y$ is non-empty.
    It is also $G$-invariant, under the natural action $g\cdot (x, H) = (g\cdot x, g \cdot H)$.
    Let $Y_x \subset \PP(V^*)$ be the fiber over $x$.
    It is also non-empty, because $G$ acts on $X$ transitively, and it is $P$-invariant.
    Since the Lie algebra of $P$ contains all positive root spaces,
       and $\hat x$ is the highest weight space,
       the line $\hat x$ is contained in every $P$-invariant
       linear subspace of $V$ (see, e.g., \cite[Prop.~14.13]{FH}).

    Fix $H_0 \in Y_x$ and consider
      the intersection $B:= \bigcap_{p \in P} p \cdot H_0$.
    This is a linear subspace of $V$, which is invariant under $P$.
    So either $B=0$ or $\hat x \subset B$. The latter is however impossible, as $\hat x \cap H_0 =0$ by our assumptions.
    So $B=0$.
    The set of hyperplanes $\set{p \cdot H_0 \in \PP V^* \mid  p \in P}$ is non-empty, irreducible with trivial base locus,
       so its dimension is positive and by a trivial instance of Bertini's Theorem there exists at least one hyperplane $H$
       in this set that avoids all points $\fromto{y_1}{y_{r-1}}$.
\end{proof}

Since there are only finitely many non-zero Fubini forms,  the parameterization:
\begin{align*}
  \phi \colon T & \to \hat X\\
        \prj{v} & \mapsto \baseptaff + \prj{v} + II(\prj{v}^2) + \dotsb
\end{align*}
is polynomial.

\begin{remark}\label{rem_every_pt_is_in_infty_or_in_image}
  Suppose $X$ is the closure of the image of a map
  \begin{align*}
  \phi \colon T & \to \PP V \\
        \prj{v} & \mapsto \baseptaff + \prj{v} + \nrml{v}(\prj{v})
  \end{align*}
  with $V = \baseptline \oplus T \oplus N$, $\baseptaff  \in \baseptline \setminus \set{0}$,
  and a polynomial map $\nrml{v} : T \to N$.
  Then every point $y \in X$ is either on the hyperplane $\PP(T\oplus N)$,
   or is in the image of the parameterization $\phi$.
\end{remark}

\begin{proof}
  We use the following elementary topological statement:
  Let $P$ be a topological space, let $I \subset U \subset P$
    with $I$ closed in $U$, and let $\bar{I}$ be the closure of $I$ in $P$.
  Then $\bar{I} \cap U = I$.
  To prove this, let $J \subset P$ be a closed subset such that $U \cap J = I$,
    which exists from the definition of subspace topology.
  Then $\bar{I} \subset J$, from the definition of the closure, and so
  \[
    I \subset \bar{I} \cap U \subset J \cap U = I.
  \]

  We use the statement with $P = \PP V$, $U$ the affine piece of $\PP V$,
    which is the complement of the hyperplane $\PP(T\oplus N)$,
    and $I = \phi(T)$.
  Note that $\phi(T)$ is closed in $U \simeq T \oplus N$, because it is the graph of $\nrml{v}$
    (which is a polynomial map by our assumption).
  Moreover,  $\bar{I} = X$, and so $X\cap U = I$, and $X \subset I \cup \PP(T\oplus N)$ as claimed.
\end{proof}

This implies  the following property of tangent spaces on $X$.

\begin{lemma}\label{lem_tangent_at_line}
    Let $X$ be generalized cominuscule and let $\ell \subset X$ be a line.
    Then the space  $\ccT^{\ell}:= { \hat T}_{[\inV{\xi}]} X + {\hat  T}_{[\inV{\eta}]} X$
      for any $[\inV{\xi}], [\inV{\eta}] \in \ell$ is independent of the choice of $[\inV{\xi}], [\inV{\eta}]$.
    Moreover, $\dim \ccT^{\ell}$ is constant over each irreducible component
      of the space parameterizing lines on $X$.
\end{lemma}

\begin{proof}
    Fix $\basept:=[\inV{\xi}] \in \ell$.   By Lemma~\ref{lem_can_pick_splitting} we may
    choose   a splitting \eqref{equ_splitting_for_gen_cominuscule} such that $[\inV{\eta}] \notin T\oplus N$.
     Thus $[\inV{\eta}]$ is in the image
            of the parameterization by Remark~\ref{rem_every_pt_is_in_infty_or_in_image}.
     Consider    a curve $\inV{y}(t) \in \hat X$ with $\inV{y}(0) = \inV{\eta}$.
    Note that $\prj{y}(0) \in T$ is  in   the tangent direction to $\ell$.
    Then in the splitting \eqref{equ_splitting_for_gen_cominuscule}:
    \begin{align*}
       \inV{y}'(0) & = \frac{\ud}{\ud t} (\baseptaff + \prj{y}(t) + II(\prj{y}(t)^2) + III(\prj{y}(t)^3) + \dotsb)|_{t=0}\\
             & = \prj{y}'(0) + 2 II(\prj{y}'(0) \prj{y}(0)) + 3 III(\prj{y}'(0) \prj{y}(0)^2) + \dotsb\\
             & \stackrel{(\star)}{=} \prj{y}'(0) + 2 II(\prj{y}'(0) \prj{y}(0)).
    \end{align*}
    Here $(\star)$ holds by the prolongation property \eqref{equ_prolongation}, because $II( \prj{y(0)}^2)  =0$.
    Thus letting $\nu'$ be any non-zero vector in $T_{\xi}\ell \subset T$ we have:
    \begin{equation}\label{equ_formula_for_ccTell}
      \ccT^{\ell} = { \hat T}_{[\inV{\xi}]} X + {\hat  T}_{[\inV{\eta}]} X =
       \set{\inV{\xi'} + II( \zeta' \nu' ) \mid \text{for } \inV{\xi'} \in  {\hat T}_{[\inV{\xi}]} X, \inV{\zeta'} \in  {T}_{[\inV{\xi}]} X}.
    \end{equation}
    This formula   is independent of $\inV{\eta}$, so we can vary $\inV{\eta}\in \ell$ freely.
    Exchanging the roles of $\inV{\xi}$, and $\inV{\eta}$, we can also vary $\inV{\xi}$.

    Thus, $\ccT^{\ell}$ is determined by the geometry of   $\ell\subset X$.
    But the group of automorphisms of $X$ acts transitively on  each  irreducible component   of the space parameterizing lines on $X$.
When $X=G/P$ with $G$ simple, this is
        \cite[Thm.~4.3]{LM0} and \cite{MR1443819}.
    (This is true for any minimally embedded homogeneous variety $G/P_I$, with $G$ simple, where $I$ indexes the
       deleted simple roots, as long as $I$ does not contain an \lq\lq exposed short root\rq\rq\ in the language of \cite{LM0}.)
When $X=Seg(v_{d_1}(G_1/P_1)\ctimes v_{d_n}(G_n/P_n))$ is generalized cominuscule (with each $G_i/P_i$ cominuscule), the set of lines
on $X$ is the disjoint union of the variety  of lines on each $G_i/P_i$ such that  $d_i=1$.
    Thus $\dim \ccT^{\ell}$ must be constant over these irreducible components.
\end{proof}

Lemma \ref{lem_tangent_at_line} allows an alternative interpretation of the points of type \ref{item_main_thm_normal_form_two_tangents}:

\begin{lemma}\label{lem_alternative_iv}
With the notation as in Theorem~\ref{thm_points_in_Sigma},
  let $\ptstypeiv$ denote  the set of points of type~\ref{item_main_thm_normal_form_two_tangents}.
   Then $[\inV{p}]\in \ptstypeiv$ if and only if
\begin{enumerate}
    \renewcommand{\theenumi}{(iv')}
    \item  \label{nitem_points_in_Sigma_case_closure_of_IIb}
           $\inV{p} = \inV{\xi'} + II( \zeta' \nu' )$ for some $\inV{\xi} \in \hat X$,
           $\inV{\xi'} \in  {\hat T}_{[\inV{\xi}]} X$,  $\inV{\zeta'}, \inV{\nu'} \in  {T}_{[\inV{\xi}]} X$
           with  $II( (\nu')^2  )=0$, i.e.,  $\nu'$ is tangent to a line on $X$ through $\xi$.
\end{enumerate}
\renewcommand{\theenumi}{(\arabic{enumi})}
Furthermore, $\ptstypeiv$ is a closed subset of $\PP V$.
\end{lemma}

\begin{proof}
   The alternative description \ref{nitem_points_in_Sigma_case_closure_of_IIb}
     follows from \eqref{equ_formula_for_ccTell}.

  To see that  $\ptstypeiv$ is a closed subset of $\PP V$,
     note $\ptstypeiv$ is the image of a projective space bundle over the variety parameterizing lines on $X$,
     whose fiber over $\ell \subset X$ is $\PP (\ccT^{\ell})$.
  Since $\dim \ccT^{\ell}$ is locally constant by Lemma~\ref{lem_tangent_at_line},
    this bundle is a projective variety,
    and thus $\ptstypeiv$ is an image of a projective variety, hence projective.
\end{proof}

In the following lemma, we provide an uniform interpretation of the points of types
  \ref{item_main_thm_normal_form_third_order_pt}--\ref{item_main_thm_normal_form_two_tangents}.

\begin{lemma}\label{lem_form_of_points_in_closure_of_II}
   $[\inV{p}]$ is of type  \ref{item_main_thm_normal_form_third_order_pt} or \ref{item_main_thm_normal_form_two_tangents},
   if and only if
   \begin{enumerate}
     \renewcommand{\theenumi}{(iii--iv)}
     \item \label{item_points_in_Sigma_case_closure_of_II}
           $\inV{p} = \inV{\xi'} + u$ for some $\inV{\xi} \in \hat X^0$,
            $\inV{\xi'} \in \hat{T}_{[\inV{\xi}]} X$,
            and  $u \in \overline{II} := \overline{\set{II(\prj{v}^2) : \prj{v} \in T}}$.
   \end{enumerate}
   \renewcommand{\theenumi}{(\arabic{enumi})}
   Moreover, for $u \in V$, the following conditions are equivalent:
   \begin{enumerate}
      \item \label{item_form_of_u_in_II_bar}
             $u \in \overline{II}$;
      \item \label{item_form_of_u_with_curve}
              There exist a curve $\prj{v}(t) \in T$ and an integer $m$,
              such that $II(\prj{v}(t)^2) = t^m u + t^{m+1} (\dotsc)$;
      \item \label{item_form_of_u_with_vectors}
              There exist an integer $m$ and vectors $\fromto{\prj{v_0}, \prj{v_1}}{\prj{v_m}} \in T$,
              such that
              \[
                 II\left(\sum_{i=0}^d \prj{v_i} \prj{v_{d-i}} \right) =
                            \begin{cases}
                                  0 & \text{if } d < m \\
                                  u & \text{if } d = m
                            \end{cases}
              \]
   \end{enumerate}
\end{lemma}
Note that $\BP \overline{II}$ is the closure of the image of the rational map $ii: \BP T\dashrightarrow \BP N$
  given by $[\prj v]\mapsto [II(\prj v^2)]$.

\begin{proof}[Proof of Lemma~\ref{lem_form_of_points_in_closure_of_II}]
   The equivalence of \ref{item_form_of_u_in_II_bar}--\ref{item_form_of_u_with_vectors} is clear.
   In the notation of \ref{item_form_of_u_with_vectors},
     a point $\inV{p}$ is of type \ref{item_main_thm_normal_form_third_order_pt}
     if and only if it is of type \ref{item_points_in_Sigma_case_closure_of_II} with $m=0$,
     and it is of type \ref{nitem_points_in_Sigma_case_closure_of_IIb}
     if and only if it is of type \ref{item_points_in_Sigma_case_closure_of_II} with $m=1$.
   So suppose $\inV{p}$ is of type  \ref{item_points_in_Sigma_case_closure_of_II} with $m>1$.
   Then it is in the closure of $\ptstypeiv$, the set of points of type \ref{nitem_points_in_Sigma_case_closure_of_IIb}.
   But $\ptstypeiv$ is closed by Lemma~\ref{lem_alternative_iv},
     so $\inV{p}$ is of type \ref{item_main_thm_normal_form_two_tangents}.
\end{proof}

\begin{proof}[Proof of Theorem \ref{thm_points_in_Sigma}]
   Suppose $\inV{p} \in \sigma_3(X)$, so
     there exist $\inV{\xi}$ and  $\inV{y}(t):=y_1(t), \inV{z}(t):=y_2(t)$  as in Lemma~\ref{lem_one_curve_is_constant}.
   Write $\inV{\xi} = \baseptaff$,
      and by Lemma \ref{lem_can_pick_splitting} we may  choose the splitting \eqref{equ_splitting_for_gen_cominuscule} such that for small values of $t$,
  we have       $\inV{y}(t),\inV{z}(t)\not\in T\op N$.
   So
      $\inV{y}(t)=(\baseptaff ,\prj{y}(t),\nrml y(t))$
      by Remark~\ref{rem_every_pt_is_in_infty_or_in_image}
      and similarly for $\inV{z}(t)$.
   Consider  the   curves $\prj{y}(t),\prj{z}(t) \in T$.
   Exchanging the roles of $\inV{y}$ and $\inV{z}$ if necessary, pick maximal integers $k,l$,
     with $l \ge k\geq 0$ and such that:
   \begin{align*}
     \prj{y}(t)& = t^k \prj{v}(t) \text{ and}\\
     \prj{z}(t)& = t^k \lambda(t) \prj{v}(t) + t^l \prj{w}(t)
   \end{align*}
   for some holomorphic function $\lambda(t) \in \CC$ and curves $\prj{v}(t), \prj{w}(t) \in T$.
   From now on,   we write $\inV{y}$ for $\inV{y}(t)$, etc.
   We adopt the convention $l=\infty$ if $\prj w=0$.

   If $l=0$, then  $0, \prj{y}_0, \prj{z}_0$ are three distinct and non-collinear points in $T$.
   This implies that $\inV{p}$ is on an honest $3$-secant,
       and we are in case \ref{item_main_thm_normal_form_honest_secant}.
   So from now on suppose $l >0$.

   Our goal is to understand the leading term  (in $t$)   of
\be\label{equ_xyz_after_gauss}
      \baseptaff \wedge \inV{y} \wedge \inV{z} =\baseptaff \wedge (\inV{y}-\baseptaff)
 \wedge (\inV{z} - \baseptaff - \lambda(\inV{y}-\baseptaff)).
\ene
Expanding out terms we obtain:
\begin{align}
\inV{y}-\baseptaff &=t^k \prj{v} + t^{2k} II( \prj{v}^2) + t^{3k} III( \prj{v}^3) + \dotsb\nonumber\\
(\inV{z} - \baseptaff &- \lambda(\inV{y}-\baseptaff))=
     t^l \prj{w} + \sum_{s=2}^{f} \FF_s\Bigl(\prj{z}^s  - \lambda \prj{y}^s\Bigr) \nonumber\\
       & = t^l \prj{w} +  \sum_{s=2}^{f} \FF_s\biggl(\left(\lambda t^k \prj{v} + t^l \prj{w}\right)^s
 - \lambda t^{sk} \prj{v}^s\biggr)\nonumber\\
        &= t^l \prj{w}_0 +  \sum_{s=2}^{f} \FF_s\Bigl((\lambda^s - \lambda) t^{sk} \prj{v}^s
              + s   \lambda^{s-1}   t^{(s-1)k+l} \prj{v}^{s-1}\prj{w})\Bigr)
              + t^{l+1} (\dotsc)\nonumber\\
     &   = t^l \prj{w}_0 +  \sum_{s=2}^f  (\lambda^s - \lambda) t^{sk} \FF_s\bigl(\prj{v}^s\bigr)
              + s \lambda^{s-1}  t^{(s-1)k+l} \FF_s\bigl(\prj{v}^{s-1}\prj{w})\bigr)
              + t^{l+1} (\dotsc).
          \label{equ_expansion_of_normal_part}
\end{align}
  First consider the case $k \ge 1$, so that the three limit points coincide: $\baseptaff =\inV{y}_0 =\inV{z}_0$.
  In this case,  the terms in \eqref{equ_expansion_of_normal_part} with $t^{(s-1)k+l}$ are of order higher than $l$.
  By Lemma~\ref{lem_vanishing_of_II_gives_vanishing_of_F_s},  the higher fundamental forms $\FF_s$ with $s\ge 3$
    will always have higher degree leading term than $II$.
  Thus:
  \[
    \baseptaff \wedge \inV{y} \wedge \inV{z} = \baseptaff \wedge t^k \prj{v_0} \wedge
      \left(t^l\prj{w_0} + t^{2k}\lambda(\lambda-1) II(\prj{v}^2)\right) + \dotsb
    \text{terms of higher order}.
  \]
  We conclude that any point $p$ in the limiting space, which is spanned by $\baseptaff$, $\prj{v_0}$, and
    the leading term of $\left(t^l\prj{w_0} + t^{2k}\lambda(\lambda-1) II(\prj{v}^2)\right)$,  is of the form
    \ref{item_points_in_Sigma_case_closure_of_II}.

\smallskip

In the remainder of the argument assume $k=0$  and we still assume $l>0$.

 If   $\lambda_0 \ne 0,1$,
     the three  limit points $ 0,  \prj{y}_0, \prj{z}_0$ are distinct,
     but they  lie on a line in $T$.
Also suppose that $II(\prj{v_0}^2) \ne 0$.
This means  (e.g. by \eqref{taylore})   that the projective line from $\basept$ in the direction of $\prj{v_0}$ is not contained in $X$.
It follows that
    $\baseptaff, \inV{y_0}, \inV{z_0}$ are linearly independent,
     because any line trisecant to $X$ is entirely contained in $X$.
    This  leads to   case \ref{item_main_thm_normal_form_honest_secant}.

 Now say  $\lambda_0 = 0$ or $1$, and $II(\prj{v_0}^2) \ne 0$.
  If $\lambda_0= 0$, then   $ \baseptaff=\inV{z}(0)$.
  If $\lambda_0= 1$, then   $\inV{y}(0)= \inV{z}(0)$.
  Swapping the roles of $\inV{x}$ and $\inV{y}$ if necessary, we may assume $\lambda_0=0$
     and write $\lambda = t^m \lambda_m + t^{m+1} (\dotsc)$ with $m\ge 1$ and $\lambda_m \ne 0$.
  Note also $\prj{y} = \prj{v}$ in this case (because $k=0$).
  Then the leading term of \eqref{equ_expansion_of_normal_part} is the leading term of
     $t^l\prj{w}_0 + \sum_{s=2}^{f} (\lambda^s - \lambda) \FF_s\bigl(\prj{y_0}^s\bigr)$
     or it is of order at least $l+1$. Therefore:
  \begin{align*}
     \baseptaff \wedge \inV{y} \wedge \inV{z} &=
        \baseptaff \wedge \inV{y_0} \wedge
        \left(t^{l} \prj{w_0} + \sum_{s=2}^{f} (\lambda^s - \lambda) \FF_s\bigl(\prj{y_0}^s\bigr)\right)
         +   \text{ terms of higher order}\\
     &= \baseptaff \wedge \inV{y_0} \wedge
        \Bigg(t^{l} \prj{w_0} +
          \underbrace{\sum_{s=2}^{f} \lambda^s \FF_s\bigl(\prj{y_0}^s\bigr)}_{=t^{2m}\cdot(\dotsc)} -
          \lambda \underbrace{\sum_{s=2}^{f}  \FF_s\bigl(\prj{y_0}^s\bigr)}_{=\inV{y_0} - \baseptaff - \prj{y}_0}\Bigg)
           +   \text{ terms of higher order}\\
     &= \baseptaff \wedge \inV{y_0} \wedge
        \left(\lambda \prj{y_0} + t^{l} \prj{w_0} \right)
           +   \text{ terms of higher order}\\
     &= \baseptaff \wedge \inV{y_0} \wedge
        \left(\lambda_m t^m \prj{y_0} + t^{l} \prj{w_0}\right)    +   \text{ terms of higher order}.
  \end{align*}
  Note $\inV{y_0}$ is linearly independent from $T$, because $II(\prj{y_0}^2) \ne 0$.
 We cannot have $m = l$ and $\prj{w_0} = - \lambda_m  \prj{y_0}$, because  then the choice of $l$ would  not
be maximal.
  Thus we have non-zero terms of degrees $l$ or $m$,
     and the limiting space is spanned by $\baseptaff, \inV{y_0}$ and a tangent vector to $\basept$
     (which is a linear combination of $\prj{y_0}$ and $\prj{w_0}$).
  Therefore we are in   case \ref{item_main_thm_normal_form_point_plus_tangent}.

 Finally,  suppose   $II(\prj{v_0}^2) = 0$
   (so the line $\langle \basept, y(0)\rangle$ is contained in $X$).

Hence \eqref{equ_expansion_of_normal_part} becomes:
\[
        t^l \prj{w}_0 +  \sum_{s=2}^f  \bigl( (\lambda^s - \lambda) \FF_s\bigl(\prj{v}^s\bigr)
              + s \lambda^{s-1}  t^{l} \FF_s\bigl(\prj{v}^{s-1}\prj{w})\bigr)
              + t^{l+1} (\dotsc).
\]
We claim that the summands with $\FF_s$ for $s \ge 3$ are irrelevant to the leading term.
First note for $s\ge 3$ the fundamental form  $ \FF_s (\prj{v}^{s-1}\prj{w})$ vanishes at $t=0$
  by the prolongation property \eqref{equ_prolongation}.
So $ t^{l} \FF_s\bigl(\prj{v}^{s-1}\prj{w})\bigr)$ has order of vanishing at least $l+1$, unless $s=2$.
Next we treat
\begin{align*}
   \sum_{s=2}^f  (\lambda^s - \lambda) \FF_s\bigl(\prj{v}^s\bigr) &
     = (\lambda^2 - \lambda) \sum_{s=2}^f (1 + \lambda + \dotsb + \lambda^{s-2}) \FF_s\bigl(\prj{v}^s\bigr)
\end{align*}
By Lemma~\ref{lem_vanishing_of_II_gives_vanishing_of_F_s},
   for $s\ge 3$ the leading term of $\FF_s(\prj{v}^s)$
    is of higher order than that of $II(\prj{v}^s)$.
Thus the leading term of \eqref{equ_expansion_of_normal_part} can only come from the leading term of
\begin{equation} \label{equ_leading_term_in_last_case}
        t^l \prj{w}_0 +  (\lambda^2 - \lambda) II\bigl(\prj{v}^2\bigr)
              + 2 \lambda  t^{l} II \bigl(\prj{v}\prj{w}_0\bigr).
\end{equation}
Suppose $\mu$ is a holomorphic function in one variable, and $m$ is the maximal integer such that
   $\lambda -1 =  t^m \mu^2$ for sufficiently small values of $t$. Note that $\mu$ has invertible values near $t=0$.
If $m \ge l$, then only  $t^l \prj{w}_0  +  2 \lambda  t^{l} II (\prj{v}\prj{w}_0)$
  contributes to the leading term of \eqref{equ_expansion_of_normal_part},
  and $p$ is of type \ref{item_points_in_Sigma_case_closure_of_II}.
Suppose $m <l$, and rewrite \eqref{equ_leading_term_in_last_case}, up to terms of order $> l$:
\[
        t^l \prj{w}_0 + \lambda t^{m}II \bigl((\mu \prj{v} + \frac{t^{l-m}}{\mu} \prj{w}_0)^2\bigr)
\]
Thus there exists $u \in \overline{II}$
    (either $u=0$ or $u$ is the leading coefficient of
    $II \bigl((\mu \prj{v} + \frac{t^{l-m}}{\mu} \prj{w}_0)^2\bigr)$ up to scale,
      compare with Lemma~\ref{lem_form_of_points_in_closure_of_II}\ref{item_form_of_u_with_curve}),
    such that the limiting space $\lim_{t\to 0} \langle \baseptaff, \inV{y}(t), \inV{z}(t)\rangle$
    is spanned by either $\baseptaff, \inV{y_0}, u$ or $\baseptaff, \inV{y_0}, \prj{w_0} + u$.
   Since $\inV{y_0} \in \hat\basept \oplus T$,
      in either case we have $p= \xi' + u$ for some $\xi'\in \hat\basept \oplus T$,
      a linear combination of $\baseptaff$, $\inV{y_0}$  and $\prj{w_0}$, and also after possible rescaling of $u$.
   That is, $p$ is a point of type \ref{item_points_in_Sigma_case_closure_of_II}.

\smallskip

  It remains to prove  that any point $p$ of the form
    \ref{item_main_thm_normal_form_honest_secant}, \ref{item_main_thm_normal_form_point_plus_tangent},
    or \ref{item_points_in_Sigma_case_closure_of_II}
    is in $\sigma_3(X)$.
 Case \ref{item_main_thm_normal_form_honest_secant} is clear,
   case \ref{item_main_thm_normal_form_point_plus_tangent}
   follows as $\s_3(X)=J(X,\s_2(X))\supset J(X,\t(X))$
 and  points on tangent lines are handled by Proposition  \ref{prop_case_r=2}.

  Finally, for   case \ref{item_points_in_Sigma_case_closure_of_II},
   take   $\xi = \baseptaff$, and $\xi' = \baseptaff+ \prj{w_0}$ with $\prj{w_0} \in T$.
  For $u \in \overline{II}$,
  let $\prj{v}$ and $m$ be as in Lemma~\ref{lem_form_of_points_in_closure_of_II}\ref{item_form_of_u_with_curve}.
  Set:
  \begin{align*}
    \inV{x}(t) &:= \baseptaff,\\
    \inV{y}(t) &:= \baseptaff + t \prj{v} + t^2 II(\prj{v}^2) + \dotsb, \text{ and}\\
    \inV{z}(t) &:= \baseptaff + 2t \prj{v} + 4t^2 II(\prj{v}^2) + \dotsb + 2 t^{m+2} \prj{w_0} + \dotsb
  \end{align*}
    i.e. $\prj{y}(t) = t \prj{v}$ and $\prj{z}(t) = 2t \prj{v} + 2 t^{m+2} \prj{w_0}$.
  We calculate:
  \[
     \inV{x}(t) \wedge \inV{y}(t) \wedge\inV{z}(t) = \baseptaff \wedge t \prj{v} \wedge t^{m+2}(2 \prj{w_0} + 2 u)
          + \dotsb \text{terms of higher order}.
  \]
  Here  $\xi' + u = \baseptaff+\prj{w_0} + u$ is in the limiting space.
\end{proof}

\section{Examples}\label{s3exsect}

In the next sections we treat   the case of Segre product with at least $3$ factors in detail.
Here we briefly review some other cases.

\subsection{Known results}
We record the following known results:

\begin{example} Let $X\subset \BP V$ be one of
  $v_2(\pp n)$ (symmetric matrices of rank one), $G(2,n)$
(skew-symmetric matrices of rank two), $Seg(\BP A\times \BP B)$
(matrices of rank one), or the Cayley plane $\BO \pp 2$. Then
any point on $\s_r(X)$ for any $r$ is on an honest secant
$\pp{r-1}$.
\end{example}

\begin{example}\cite{LTrank} Let $X=v_d(\pp n)$ for $d>2$.
Then any point in $\s_3(X)$ is of  the form
\renewcommand{\theenumi}{(\roman{enumi})}
\renewcommand{\labelenumi}{\theenumi}
\begin{enumerate}
    \item
           $\inV{p} = \inV{\xi} + \inV{\eta} +\inV{\zeta}$
           for some $\inV{\xi}, \inV{\eta}, \inV{\zeta} \in \hat X$ ($\inV{p}$ is on an honest $3$-secant plane), or
    \item
           $\inV{p} = \inV{\xi'} +\inV{\eta}$ for some $\inV{\xi}, \inV{\eta} \in \hat X$
           and $\inV{\xi'} \in  {T}_{[\inV{\xi}]} X$, or
    \item
           $\inV{p} = \inV{\xi'} + II(\eta',\eta')$ for some $\inV{\xi} \in \hat X$,
           $\inV{\xi'}, \inV{\eta'} \in  {T}_{[\inV{\xi}]} X$,
   \end{enumerate}
Normal forms for  $\s_3(v_d(\BP V))\backslash
\s_2(v_d(\BP V))$ of these types are respectively $x^d+y^d+z^d$, $x^{d-1}y+z^d$ and
$x^{d-1}y+ x^{d-2}z^2$, where $x,y,z\in V$.
Thus the  points  of type \ref{item_main_thm_normal_form_two_tangents} do   not occur in this case.
\end{example}
\renewcommand{\theenumi}{(\arabic{enumi})}
\renewcommand{\labelenumi}{\theenumi}

The generalized cominuscule varieties with $\s_2(X)=\BP V$ are
$Seg(\pp 1 \times \pp n)$, $Seg(\pp 1\times \pp 1\times \pp 1)$, quadric hypersurfaces $Q$,
the Veronese varieties
$v_2(\pp 1)$, $v_3(\pp 1)$,
the Grassmannians  $G(2,5)$ and $G(3,6)$, the spinor varieties $\BS_5$ and $\BS_6$,
the Lagrangian Grassmannian $G_{Lag}(3,6)$, $Seg(\pp 1\times Q)$, and the Freudenthal variety
$E_7/P_7$.

\subsection{Grassmannians in Pl\"ucker embedding}

Let $X:=G(k,n) \subset \PP(\Wedge{k} \CC^n)$, and suppose $3\leq k\leq n-k$ and $n-k>3$.
The tangent space  at   $E\in G(k,n)$
    can be identified with the space of  $k \times (n-k)$-matrices $\Wedge{k-1}E \otimes F \simeq E^* \otimes F$,
    where $F = \CC^n/E$.
The local parametrization in this case comes from a choice of splitting $\CC^n \simeq E \oplus F$
    and the determined splitting:
    \begin{alignat*}{7}
       \Wedge{k} (E\oplus F) &=&\ \Wedge{k} E \ & \oplus & \ \Wedge{k-1} E &\otimes F \  &&\oplus & \ \Wedge{k-2} E & \otimes \Wedge{2} F
                &&\oplus \dotsb \oplus &   & \phantom{\otimes } \ \Wedge{k} F \\
                             &\simeq& \baseptline \ &\oplus &E^*& \otimes F  &&\oplus &\Wedge{2} E^* &\otimes \Wedge{2} F
                &&\oplus \dotsb \oplus & \ \Wedge{k} E^* & \otimes \Wedge{k} F.
    \end{alignat*}
The parametrization has the following form:
\[
  T \simeq E^* \otimes F \ni M \stackrel{\varphi}{\mapsto}
  [\underbrace{1,}_{\in \baseptline} \underbrace{M,}_{\in T} \underbrace{\Wedge{2} M,}_{=II(M^2)\in N_2} \dotsc,
  \underbrace{\Wedge{k} M}_{\in N_k}],
\]
where $\FF_s(M^s) = \Wedge{s} M \in \Wedge{s} E^* \otimes \Wedge{s} F$,  expressed in linear coordinates,  is the collection of all $s\times s$ minors of $M$.

In the normal forms of Theorem~\ref{thm_points_in_Sigma} we can take the first point $\xi = \baseptaff$,
  for the second we have $k$ choices given the rank of $M$.
Let $\epsilon_i$ for $i \in \setfromto{1}{k}$
  denote the matrix of rank $i$ with the block form
$
 \begin{pmatrix}
   \Id_i & 0\\
     0 & 0
 \end{pmatrix}
$.
The normal forms are:
  \begin{itemize}
    \item[\ref{item_main_thm_normal_form_honest_secant}]
           $\inV{p} = \baseptaff + \varphi(\epsilon_i) + \varphi(M)$ for some $i$, $M$,
    \item[\ref{item_main_thm_normal_form_point_plus_tangent}]
           $\inV{p} = \baseptaff + M+ \varphi(\epsilon_i) $ or  $\inV{p} = M + \varphi(\epsilon_i)$ for some $i$, $M$,
    \item[\ref{item_main_thm_normal_form_third_order_pt}]
           $\inV{p} = \baseptaff +  M + \Wedge{2}\epsilon_i$  or  $\inV{p} = M + \Wedge{2}\epsilon_i$ for some $i$, $M$,
    \item[\ref{nitem_points_in_Sigma_case_closure_of_IIb}]
           $\inV{p} = \baseptaff +  M + \Wedge{2}\epsilon_{i+1} -\Wedge{2}\epsilon_{i}$
                                        or  $\inV{p} = M + \Wedge{2}\epsilon_{i+1} -\Wedge{2}\epsilon_{i}$
                                        for some $i \ne k$, $M$.
  \end{itemize}
In \ref{nitem_points_in_Sigma_case_closure_of_IIb}, $\nu = \epsilon_{i+1} -\epsilon_{i}$ is a rank $1$ matrix,
    so $II(\nu^2)=0$, and $\Wedge{2}\epsilon_{i+1} -\Wedge{2}\epsilon_{i} =  \frac{1}{2} II (u, \epsilon_i)$.
In all normal forms, we can pick $M$ to be in some normal form. For example, if $i=k=n-k$,
   then $M$ may be
   (at least)
   assumed to be in Jordan normal form.

\subsection{Lagrangian Grassmannians}
Let $X$ be the Lagrangian Grassmannian $G_{Lag}(k,2k)=C_k/P_k \subset \PP (V_{\omega_k})$
 with $k>3$, where $V_{\omega} = \Wedge{k} \CC^{2k} /\Wedge{k-2} \CC^{2k}$ is the minimal homogeneous embedding.
In this case the local parametrization is identical, but with $T \simeq S^2 \CC^k$
    and $M$ a  symmetric $k \times k$ matrix,
    see \cite[\S 5]{boralevi_jabu_secants_to_LG}.
The normal forms are also identical.

\subsection{Spinor varieties}
Let $X$ be  the    spinor variety $\BS_{k}=D_k/P_k$ for $k\ge 7$
   in its minimal homogeneous embedding $\PP(\Wedge{even} \CC^{k})$.
In this case  $T\simeq \Wedge{2} \CC^k$ and $M$ is a skew-symmetric $k\times k$ matrix,
and the parameterization  is   similar to the previous cases:
\[
  M \stackrel{\varphi}{\mapsto}
  [\underbrace{1,}_{\in \baseptline} \underbrace{M,}_{\in T}  \ \underbrace{\Pf_4 M,}_{=II(M^2)\in N_2} \
   \underbrace{\Pf_6 M,}_{=III(M^3)\in N_3}\dotsc],
\]
where $\Pf_{2s} M\in \Wedge{2s} \CC^k$,  expressed in linear coordinates,   is the collection of all $2s\times 2s$  sub-Pfaffians   of $M$.

Let $\epsilon^{skew}_i$ for $i \in \setfromto{1}{\lfloor \frac{1}{2} k \rfloor}$
  denote the matrix of rank $2i$ with the block form
$
 \begin{pmatrix}
   0      & \Id_i & 0 \\
   -\Id_i & 0     & 0 \\
   0      & 0     & 0
 \end{pmatrix}
$.
The normal forms are:
  \begin{itemize}
    \item[\ref{item_main_thm_normal_form_honest_secant}]
           $\inV{p} = \baseptaff + \varphi(\epsilon^{skew}_i) + \varphi(M)$ for some $i$, $M$,
    \item[\ref{item_main_thm_normal_form_point_plus_tangent}]
           $\inV{p} = \baseptaff + M+ \varphi(\epsilon^{skew}_i) $ or  $\inV{p} = M + \varphi(\epsilon^{skew}_i)$ for some $i$, $M$,
    \item[\ref{item_main_thm_normal_form_third_order_pt}]
           $\inV{p} = \baseptaff +  M + \Pf_4\epsilon^{skew}_i$  or  $\inV{p} = M + \Pf_4\epsilon^{skew}_i$ for some $i$, $M$,
    \item[\ref{nitem_points_in_Sigma_case_closure_of_IIb}]
           $\inV{p} = \baseptaff +  M + \Pf_4\epsilon^{skew}_{i+1} -\Pf_4\epsilon^{skew}_{i}$
                                        or  $\inV{p} = M + \Pf_4\epsilon^{skew}_{i+1} -\Pf_4\epsilon^{skew}_{i}$
                                        for some $i \ne \lfloor \frac{1}{2} k \rfloor$, $M$.
  \end{itemize}

\section{The Segre product  $Seg(\BP A_1\ctimes \BP A_n))$}\label{smallsecsegsect}

Recall that for any smooth variety $X$, if $x\in \s_2(X)$, then either
$x\in X$, $x\in \s_2^0(X)$ or $x$ lies on an embedded tangent line to $X$, see Proposition~\ref{prop_case_r=2}.

\subsection{Proof of Proposition \ref{lastthm}}

  All the assertions except for the rank
of $x$ in \eqref{tprimevect}  are immediate.
The rank of $x$ is at most  $|J|$ because there
are $|J|$ terms in the summation.

Assume without loss of generality  $|J|=n$ and  work by induction.  The case $n=2$ is clear.
Now assume we have established the result up to $n-1$, and consider
   $x(A_1^*)$. It is spanned by
$$
a_1^2\otc a_1^n,
\sum_{j=2}^n a_1^2\otc a_{1}^{j-1}\ot a_2^{j}\ot a_1^{j+1}\otc a_1^n.
$$
By induction, the second vector has rank $n-1$,  so the only way $x(A_1^*)$ could be spanned
by $n-1$ rank one elements would be if there were an expression of the second   vector as a sum of $n-1$ decomposable tensors
where one of terms is a multiple of  $a_1^2\otc a_1^n$. Say there were  such an expression,  where $a_1^2\otc a_1^n$
appeared with coefficient $\l$, then the tensor
$\sum_{j=2}^n a_1^2\otc a_{1}^{j-1}\ot a_2^{j}\ot a_1^{j+1}\otc a_1^n - \l a_1^2\otc a_1^n$
would have rank $n-2$, but setting $\tilde a^2_2=a^2_2-\l a^2_1$
and $\tilde a^j_2=a^j_2$ for $j \in \setfromto{3}{n}$,
this would imply that
\[
  \sum_{j=2}^n a_1^2\otc a_{1}^{j-1}\ot \tilde a_2^{j}\ot a_1^{j+1}\otc a_1^n
\]
 had rank $n-2$,
a contradiction.

\begin{remark} The case $n=3$ was previously established by Grigoriev, Ja'Ja' and Teichert.
\end{remark}

\subsection{Parameterization in the Segre case}\label{sec_parameterization_Segre}
Suppose $X = Seg(\BP A_1\ctimes \BP A_n)$.
Let $\baseptaff=a^1_1\otc a^n_1$, and let $A_j'=a^1_1\otc a^{j-1}_1\ot (A_j/a^j_1)\ot  a^{j+1}_1\otc a^n_1
\simeq A_j/a^j_1
$.
Then $T = A_1'\op\cdots \op A_n'$ and $X$ is parametrized by
\[
  (\fromto{a'_1}{a'_n}) \mapsto [\underbrace{1,}_{\in \baseptline} \underbrace{\fromto{a'_1}{a'_n}}_{\in T}, \ \underbrace{\fromto{a'_1\otimes a'_2}{a'_{n-1}\otimes a'_n}}_{=II((\fromto{a'_1}{a'_n})^2)\in N_2}, \dotsc,
  \underbrace{a'_1\otimes a'_2 \otimes\dotsm\otimes a'_n}_{\in N_n}].
\]
Thus $II((\fromto{a'_1}{a'_n})\cdot(\fromto{b'_1}{b'_n}))
     = \frac{1}{2}(\fromto{a'_1 \otimes b'_2 + b'_1 \otimes a'_2}{a'_{n-1} \otimes b'_n + b'_{n-1} \otimes a'_{n}})$.

In this case the base locus of $II$ is $\BP A_1'\sqcup \cdots \sqcup \BP A_n'\subset \BP (A_1'\op\cdots \op A_n') \simeq \PP T$.
If $II(\prj v_0^2)=0$, then $\prj v_0\in A_i'$ for some $i$ and if
further $II(\prj v_0\prj v_1)=0$ then
$\prj v_1\in A_i'$ for the same $i$.

In particular, if a line $\ell \subset X$ contains $\basept$ and is tangent to $\prj v_0$,
  then by \eqref{equ_formula_for_ccTell} we have:
\begin{equation}\label{equ_dim_ccT_ell}
  \dim \ccT^{\ell} = 2 \dim X + 1 - \dim \ker II(\prj{v}_0 \ \cdot ) = 2 \dim X +2 - \dim A_i.
\end{equation}

Now we prove Theorem~\ref{s3nformthm}.
The normal forms follow from the discussion in the previous
sections.

Now suppose $\dim A_i \ge 3$.
To see that the general points of each type do not belong to the other types,
  note that for any type and for any $i$,
  in the normal forms \eqref{nf1}--\eqref{nf4}
  either $a_1^i, a_2^i, a_3^i$  are linearly independent,
  or the point is contained in a subspace variety, i.e., a closed subvariety consisting of tensors in some
  $A_1 \otimes \dotsb \otimes A_{i-1} \otimes \CC^2 \otimes A_{i+1} \otimes \dotsb \otimes A_n$.
Thus the general points of each type form a single orbit
  (or $n$ orbits for type \ref{item_main_thm_normal_form_two_tangents}) of the action of $GL(A_1) \times \dotsm \times GL(A_n)$.
Therefore the only possible way that they could  overlap, is if one of the orbits were equal to the other.
But the orbits are distinct  by the dimension count below, which we present in two different forms.

\subsection{First proof of dimensions in Theorem~\ref{s3nformthm}}
%
%
%

We   compute the Lie algebras of the stabilizers of each type of point.
Without loss of generality (for computing codimension), assume $\tdim A_j=3$.
Write $\Gamma =(x_1\hd x_n)$ where $x_\a=(x^i_{j,\a})$, $1\leq i,j\leq 3$.
We calculate the $\Gamma $ such that $\Gamma .p_{(*)}=0$ in each case $*=i,ii,iii,iv$  and denote this algebra
by $\fg_{p_{(*)}}$. In each case one has a system of $3^n=\tdim(A_1\otc A_n)$ linear
equations, many of which are zero or redundant.

$$
\fg_{p_{(i)}}=
\left\{ \Times_{\a=1\hd n} \begin{pmatrix} x^1_{1,\a} & 0 & 0 \\ 0 & x^2_{2,\a} & 0 \\ 0 & 0 & x^3_{3,\a}\end{pmatrix}
\mid \sum_{\a}x^i_{i,\a}=0,\ i=1,2,3 \right\}.
$$

Note $\tdim \fg_{p_{(i)}}=3n-3$.

$$
\fg_{p_{(ii)}}=
\left\{\Times_{\a=1\hd n} \begin{pmatrix} x^1_{1,\a} &  x^1_{2,\a} & 0 \\ 0 & -\sum_{\b\neq \a} x^1_{1,\b} & 0 \\ 0 & 0 & x^3_{3,\a}\end{pmatrix}
\mid \sum_{\a}x^3_{3,\a}=0, \   \sum_{\a}x^1_{2,\a}=0\right\}.
$$

Note $\tdim \fg_{p_{(ii)}}=3n-2$.

$$
\fg_{p_{(iii)}}=
\left\{ \Times_{\a=1\hd n} \begin{pmatrix} x^1_{1,\a} &  x^1_{2,\a} & x^1_{3,\a}
 \\ 0 & -\sum_{\b\neq \a} x^1_{1,\b}  & -\sum_{\b\neq \a}x^1_{2,\b} \\ 0 & 0 & -\sum_{\b\neq \a}x^1_{1,\b}\end{pmatrix}
\mid  \   \sum_{\a}x^1_{3,\a}=0\right\}.
$$

Note $\tdim \fg_{p_{(iii)}}=3n-1$.

\begin{align*}
&\fg_{p_{(iv)}}=\\
&
\left\{
\begin{pmatrix} x^1_{1,1} &  x^1_{2,1} & x^1_{3,1}
 \\ x^2_{1,1} & x^2_{2,1} & -\sum_{\rho}x^1_{2,\rho}
\\ 0 & 0 & -\sum_{\rho}x^1_{1,\rho}\end{pmatrix}
,
 \Times_{\rho=2\hd n}\begin{pmatrix} x^1_{1,\rho} &  x^1_{2,\rho} & x^1_{3,\rho}
 \\ 0 & -\sum_{\s\neq \rho} x^1_{1,\s}-x^2_{2,1} & -x^1_{2,1} \\ 0 & -x^2_{1,1} & -\sum_{\b\neq \rho}x^1_{1,\b}\end{pmatrix}
\mid  \   \sum_{\a}x^1_{3,\a}=0\right\}.
\end{align*}
Here the index ranges are $1\leq \a,\b\leq n$, $2\leq \rho,\s\leq n$.
Note $\tdim \fg_{p_{(iv)}}=3n+1$.

\subsection{Second Proof of dimensions in Theorem~\ref{s3nformthm}}

Throughout this section $X= Seg(\PP A_1 \times \dotsm \times \PP A_n)$.

We show the assertion about the codimensions of types (ii),(iii),(iv).
Type \ref{item_main_thm_normal_form_point_plus_tangent} is immediate as its closure is $J(X,\t(X))$ which is easily seen to
have the expected dimension via Terracini's lemma.

We will use the following lemma:
\begin{lemma}\label{lem_scheme_deg_3}
   Suppose $n \ge 2$ and $\dim A_i \ge 3$ for all $i \in \setfromto{1}{n}$.
   Let $R$ be a degree $3$, zero dimensional subscheme of $X= Seg(\PP A_1 \times \dotsm \times \PP A_n)$.
   Suppose moreover $R$ is in general position, that is, it is not contained in any
   $Seg(\PP A_1 \times \dotsm \times \PP^1 \times \dotsm \times \PP A_n)$.
   Let $\langle R \rangle \simeq \PP^2 \subset \PP(A_1 \otimes \dotsb \otimes A_n)$
       denote the smallest linear space containing $R$.
    Then $X \cap \langle R \rangle = R$.
\end{lemma}

\begin{proof}
   Any such $R$ is isomorphic either to $3$ distinct reduced points,
     or a double point and a reduced point,
     or  one of the two kinds of  triple points:  $\Spec \CC[x]/x^3$, or $\Spec \CC[x,y]/\langle x^2, xy, y^2 \rangle$.

   If $n=2$, without loss of generality, we may suppose  $\dim A_1 = \dim A_2 = 3$.
   We can write down explicitly $\langle R \rangle \subset \PP (A_1 \otimes A_2)$ for each of the schemes as, respectively:
   \[
      \begin{pmatrix}
        s &   &   \\
          & t &   \\
          &   & u
      \end{pmatrix},
      \begin{pmatrix}
        t & s &   \\
        s &   &   \\
          &   & u
      \end{pmatrix},
      \begin{pmatrix}
        u & t & s \\
        t & s &   \\
        s &   &
      \end{pmatrix},
      \begin{pmatrix}
        t & s & u \\
        s &   &   \\
        u &   &
      \end{pmatrix}
   \]
     The  claim  may be verified explicitly for each case, by calculating the scheme defined
     by $2 \times 2$ minors of each of the matrices.

   If $n \ge 3$, let $B_i = A_1 \otimes\dotsb \otimes A_{i-1} \otimes A_{i+1} \otimes\dotsb \otimes A_n$.
   Then $X  = \bigcap_{i =1}^n\PP A_i \times \PP B_i$, and
     the  claim  easily  follows   from the $n=2$ statement.
\end{proof}

\begin{lemma}\label{lem_line}
   Suppose $n \ge 2$ and $\dim A_i \ge 3$ for all $i \in \setfromto{1}{n}$.
   Let $X= Seg(\PP A_1 \times \dotsm \times \PP A_n)$
   and let $\ell \subset X$ be a line  spanned by $x,y\in X$.
   Let $v \in \hat{T}_x X + \hat{T}_y X$ be general and consider $\PP^2$ spanned by $\ell$ and $[v]$.
          Then $\PP^2 \cap X = \ell$.
\end{lemma}
\begin{proof}
   Let $x = a^1_{1}\otc a^{n}_{1}$, $y =
   a^1_{2}
   \ot a^2_{1}\otc a^{n}_{1}$,
     and $v$ be as in \eqref{nf4}.
   Let $B:=  A_2 \otimes\dotsb \otimes A_n$ and:
   \begin{align*}
     b_1 &:=a^2_{1}\otc a^{n}_{1}, \\
     b_2 &:=\sum_{ i=2 }^n a^2_{1}\otc a^{i-1}_{1}\ot a^i_2 \ot a^{i+1}_1 \otc a^n_{1},\\
     b_3 &:=\sum_{ i=2 }^n a^2_{1}\otc a^{i-1}_{1}\ot a^i_3 \ot a^{i+1}_1 \otc a^n_{1}.
   \end{align*}
   Then $x = a^1_1 \otimes b_1$, $y =
        a^1_2 \otimes b_1$ and
        $v =  a^1_1 \otimes b_3 + a^1_2 \otimes b_2 + a^1_3 \otimes b_1 $.
   Consider a linear combination $ sv + tx + uy$.
   The intersection $\PP^2 \cap X$ is contained in the zero locus of the $2 \times 2$ minors of the following matrix:
   \[
   \begin{pmatrix}
      t &   & s \\
      u & s  &  \\
      s &   &
    \end{pmatrix},
   \]
   which can be  identified with the line $s=0$, that is the line spanned by $x$ and  $y$.
\end{proof}

\smallskip

Let $\ptstypeiii$ be the closure of the set of points of type \ref{item_main_thm_normal_form_third_order_pt}.
Let $[p] \in  \ptstypeiii$ be a general point.
We claim such $p$  uniquely determines   $[x]$ such that $p = x+ x' + x''$.
  Suppose without loss of generality $\dim A_1 =3$.
Write $p=p_{(iii)}$ of \eqref{nf3},
and consider the underlying map $p_{(iii)}:
   {A_1}^* \to
    A_{2} \otimes\dotsb \otimes A_n$:
\begin{align*}
  p_{(iii)}({a^1_1}^*) & =  \sum_{2 \le i<j}^n a^2_{1}\otc a^{i-1}_{1}\ot a^i_2 \ot a^{i+1}_1
                    \otc a^{j-1}_{1}\ot a^j_2 \ot a^{j+1}_1 \otc a^n_{1} \\
               & + \sum_{i=2}^n a^2_{1}\otc a^{i-1}_{1}\ot a^i_3 \ot a^{i+1}_1 \otc a^n_{1},\\
  p_{(iii)}({a^1_2}^*) & = \sum_{j=2}^n a^1_{1}\otc a^{j-1}_{1}\ot a^j_2 \ot a^{j+1}_1 \otc a^n_{1},\\
  p_{(iii)}({a^1_3}^*) & = a^2_{1}\otc a^n_{1}.
\end{align*}
The projectivization of the image is a $\PP^2$ containing
  a degree $3$ scheme $R \subset Seg(\PP A_{2} \times\dotsb \times \PP A_n)$
  in general position,
  which is isomorphic to the triple point $\Spec \CC[x]/x^3$ point supported at $[p({a^1_3}^*)]$.
By Lemma~\ref{lem_scheme_deg_3},
  $R$ is determined by $\langle R \rangle =  \PP( p({A_1}^*))$,
  so it is  independent of the choice of normal form.
Therefore $\langle {a^1_2}^*, {a^1_3}^* \rangle$, which is the linear span of the unique degree $2$ subscheme of $R$,
   is determined by $p$, and so is $a^1_1$ (up to scale).
Similarly, $a^i_1$ are determined by $p$ up to scale.

Thus we have a rational dominant map
  $\psi: \ptstypeiii \dashrightarrow X$,
  $\psi(p) := [a^1_{1}\otc a^n_{1}]$.
A general fiber over $\basept \in X$ is contained in the second osculating space
   $\PP(\baseptline \oplus T \oplus N_2)$,
   and its closure is equal to the closure of points of the form $\baseptaff + \prj{\xi'} + II(\prj{v}^2)$.
Thus $\dim \ptstypeiii = 3 \sum (\dim A_i-1)$.

\smallskip

Finally consider $\ptstypeiv$, the set of points of type \ref{item_main_thm_normal_form_two_tangents}, which is
closed by Lemma \ref{lem_alternative_iv}.
Let $[p] \in  \ptstypeiv$ be a general point of any of the irreducible components.
We claim   $p$  uniquely determines the line $\PP \langle x,y  \rangle$
   such that $p = x+ x' + y + y'$.
Suppose without loss of generality $\dim A_i =3$ for all $i$.
Possibly permuting the factors, write $p=p_{(iv)}$ of  \eqref{nf4}.
First consider the underlying map $p_{(iv)}:{A_1}^*: \to  A_{2} \otimes\dotsb \otimes A_n$:
\begin{align*}
  p_{(iv)}({a^1_1}^*) & = \sum_{ i=2 }^n a^2_{1}\otc a^{i-1}_{1}\ot a^i_3 \ot a^{i+1}_1 \otc a^n_{1},\\
p_{(iv)}({a^1_2}^*) & = \sum_{ i=2 }^{n} a^2_{1}\otc a^{i-1}_{1}\ot a^i_2 \ot a^{i+1}_1 \otc a^n_{1},\\
  p_{(iv)}({a^1_3}^*) & = a^2_{1}\otc a^n_{1}.
\end{align*}
The projectivization of the image is a $\PP^2$ containing
  a degree $3$ scheme $R \subset Seg(\PP A_{2} \times\dotsb \times \PP A_n)$
  in general position,
  which is isomorphic to the triple point $\Spec \CC[x,y]/\langle x^2, xy, y^2 \rangle$ point supported at $[p_{(iv)}({a^1_3}^*)]$.
By Lemma~\ref{lem_scheme_deg_3},
  $[p_{(iv)}({a^1_3}^*)]$ is the unique reduced point in $\PP( p_{(iv)}({A_1}^*)) \cap Seg(\PP A_{2} \times\dotsb \times \PP A_n)$,
  so independent of the choice of normal form.
Therefore $\langle a^1_1, {a^1_2} \rangle \subset \PP A_1$ is determined by $p_{(iv)}$.

Now consider  $p_{(iv)}:{A_n}^*: \to  A_{1} \otimes\dotsb \otimes A_{n-1}$:
\begin{align*}
  p_{(iv)}({a^n_1}^*) & = \sum_{ i=1 }^{n-1} a^1_{1}\otc a^{i-1}_{1}\ot a^i_3 \ot a^{i+1}_1 \otc a^{n-1}_{1}\\
               & + \sum_{ i=2 }^{n-1} a^1_{2}\ot a^2_{1}\otc a^{i-1}_{1}\ot a^i_2 \ot a^{i+1}_1 \otc a^{n-1}_{1}.\\
 p_{(iv)}({a^n_2}^*) & = a^1_{2}\ot a^2_{1}\otc a^{n-1}_{1},\\
 p_{(iv)}({a^n_3}^*) & = a^1_{1}\otc a^{n-1}_{1}.
\end{align*}
By Lemma~\ref{lem_line} the projective line  $\PP \langle p_{(iv)}({a^n_2}^*),p_{(iv)}({a^n_3}^*) \rangle$ is determined by $p_{(iv)}$.
Thus $a^n_1$ (and similarly $a^i_1$ for $i \in \setfromto{2}{n}$) is determined (up to scale) by $p_{(iv)}$.
Therefore, the line $\PP(\langle a^1_1,
 a^1_2 \rangle \otimes a^2_1 \otimes \dotsb \otimes a^n_1) \subset X$
  is uniquely determined by $p_{(iv)}$.

The lines on $X$ are parametrized by $n$ irreducible varieties:
\[
  L_i:= \PP A_1 \times\dotsb \times  \PP A_{i-1} \times G(2, A_i) \times \PP A_{i+1} \times \dotsb \times \PP A_n.
\]
By the argument above we have a rational dominant map
  $\chi: \ptstypeiv \dashrightarrow L_1\sqcup \dots\sqcup L_n$.
A general fiber over $\ell \in L_i$ is $\PP \ccT^{\ell}$ in the notation of Lemma~\ref{lem_tangent_at_line},
    the linear span of projective tangent spaces to $X$ at points of $\ell$.
By \eqref{equ_dim_ccT_ell}   $\dim \ccT^{\ell} = 2 \dim X + 2 - \dim A_i$,
   and the dimension of each irreducible component of $\ptstypeiv$ is equal to $3 \sum (\dim A_i-1) -2$.

\section{Orbits of tensors in $A\otimes B \otimes C$ of border rank at most $3$}\label{sec_ranks_and_orbits}
Let $A\simeq \CC^{\aaa}$, $B\simeq \CC^{\bbb}$, $C\simeq \CC^{\ccc}$.
Let
\begin{multline*}
    Sub_{\aaa',\bbb',\ccc'}= Sub_{\aaa',\bbb',\ccc'}(A\ot B\ot C)=\\
\{T\in A\ot B\ot C\mid \exists \BC^{\aaa '}\subset A, \ \BC^{\bbb '}\subset B,\ \BC^{\ccc '}\subset C, \rm{\ such \ that \ }
T\in  \BC^{\aaa '}\ot  \BC^{\bbb '}\ot  \BC^{\ccc '}\}
\end{multline*}
 This {\it subspace variety} admits a desingularization as follows. Let $\mathcal E\ra G(\aaa',A)\times G(\bbb',B)\times G(\ccc',C)$
be $\mathcal E=\cS_{A}  \boxtimes \cS_{B}\boxtimes  \cS_{C}$, where $\cS_{A}\ra G(\aaa',A)$ is the tautological rank $\aaa'$ subspace bundle and similarly
for $B,C$. Then $\BP \mathcal E \ra Sub_{\aaa',\bbb',\ccc'}(A\ot B\ot C)$ is a desingularization and using it one can see that
$Sub_{\aaa',\bbb',\ccc'}(A\ot B\ot C)_{sing}=
Sub_{\aaa'-1,\bbb',\ccc'}  \cup Sub_{\aaa',\bbb'-1,\ccc'} \cup Sub_{\aaa',\bbb',\ccc'-1}
$,    whenever $\aaa' < \bbb \ccc$, and similarly for permuted statements.
In \cite[\S 6]{BLtensor}, normal forms for tensors in $Sub_{233}\cup Sub_{323}\cup Sub_{332}$ are given. There are
$33$ such.

We present the list of remaining orbits in $\sigma_3(Seg(\PP A \times \PP B \times \PP C))$
  under the action of $GL(A) \times GL(B) \times GL(C)$.

Each orbit is uniquely determined  by its closure, which is an algebraic variety listed in the second column of the table.
The orbit itself is an open dense subset of this variety.
The dimension of the algebraic variety is in the third column.
The fourth column is the normal form of the underlying tensor,
   the distinct variables are assumed to be linearly independent.
The normal form is also given as a slice.
The border rank and rank are given in the next columns.

\begin{table}[htb]
$$
\begin{array}{|r|c|c|ll|c|c|}
\hline
\hline
\#&\text{orbit closure}& \dim &\text{normal form} & \text{slice}  &  \ur &  \mathbf{R}  \\
\hline
\hline
34& \ptstypeiv_{A} & 3\aaa + 3\bbb+3\ccc - 11 & \begin{matrix}
           a_1\ot (b_1\ot c_2 +  b_2\ot c_1) +a_2\ot b_1\ot c_1\\
         + a_3\ot (b_3\ot c_1 + b_1\ot c_3)
        \end{matrix}&
\left(\begin{smallmatrix} t & s& u\\  s &  & \\ u &   & \end{smallmatrix} \right)
& 3&5\\
\hline
35& \ptstypeiv_{B} & 3\aaa + 3\bbb+3\ccc -11  & \begin{matrix}
          a_1\ot (b_1\ot c_2 +  b_2\ot c_1 + b_3 \ot c_3)\\
         +a_2\ot b_1\ot c_1 + a_3\ot b_3\ot c_1
        \end{matrix}&
\left(\begin{smallmatrix} t & s&\\  s &  & \\ u&   & s\end{smallmatrix} \right)
& 3&5\\
\hline
36& \ptstypeiv_{C} & 3\aaa + 3\bbb+3\ccc -11  & \begin{matrix}
          a_1\ot (b_1\ot c_2 +  b_2\ot c_1 + b_3 \ot c_3)\\
         +a_2\ot b_1\ot c_1 + a_3\ot b_1\ot c_3
        \end{matrix}&
\left(\begin{smallmatrix} t & s& u\\  s &  & \\ &   & s\end{smallmatrix} \right)
& 3&5\\
\hline
\hline
37& \ptstypeiii & 3\aaa + 3\bbb+3\ccc -9  & \begin{matrix}
    a_1\ot (b_1\ot c_3+  b_2\ot c_2 + b_3\ot c_1) \\
  + a_2\ot (b_1\ot c_2+ b_2 c_1) +   a_3\ot b_1\ot c_1
        \end{matrix}&
\left(\begin{smallmatrix} u & t&s\\   t&s  & \\ s &   & \end{smallmatrix} \right)
& 3&5 \\
\hline
\hline
38& \ptstypeii & 3\aaa + 3\bbb+3\ccc -8  & \begin{matrix}
           a_1\ot (b_1\ot c_2+  b_2\ot c_1)\\
         + a_2\ot b_1\ot c_1+ a_3\ot b_3\ot c_3
        \end{matrix}&
\left(\begin{smallmatrix} t & s&\\  s &  & \\ &   & u\end{smallmatrix} \right)
& 3&4 \\
\hline
\hline
39& \sigma_3(X) & 3\aaa + 3\bbb+3\ccc -7 & \begin{matrix}
            a_1 \ot b_1\ot c_1+a_2 \ot b_2\ot c_2
          + a_3\ot b_3\ot c_3
        \end{matrix}
& \left(\begin{smallmatrix} s & & \\   & t & \\  & & u\end{smallmatrix} \right)
& 3&3\\
\hline
\hline
\end{array}
$$
\caption{Orbits of border rank $3$ in $A\otimes B \otimes C$
         that are not contained in a $Sub_{233}$, $Sub_{323}$, or $Sub_{332}$.
         Orbits $34$--$36$ are identical up to permutations of $A$, $B$, $C$.}
\label{table_orbits_general}
\end{table}

$\ptstypeiv_A$, $\ptstypeiv_B$, $\ptstypeiv_C$
   denote the three components of $\ptstypeiv$,  the set points of type \ref{item_main_thm_normal_form_two_tangents}
   in Theorem~\ref{s3nformthm}.
$\ptstypeiii$ denotes the the closure of the set points of type \ref{item_main_thm_normal_form_third_order_pt},
   while $\ptstypeii$ denotes the the closure of the set points of type \ref{item_main_thm_normal_form_point_plus_tangent}.

The ranks of cases $34$--$37$ in Table~\ref{table_orbits_general} are calculated in \S\ref{sec_calculating_ranks}.
The rank of case $39$ is obvious, while the rank of case $38$ is at most $4$, due to the normal form expression.
If it were $3$, then a general point of type \ref{item_main_thm_normal_form_point_plus_tangent},
  would be expressible as a point of type \ref{item_main_thm_normal_form_honest_secant}, a contradiction with
  Theorem~\ref{s3nformthm}.

\subsection{Proof of Theorem  \ref{lastcor}}\label{sec_calculating_ranks}

The {\it rank} of a linear subspace $U\subset \BC^k\ot \BC^l$ is the smallest $r$ such that $U$ is contained in a linear
space of dimension $r$ spanned by rank one elements. The rank of a tensor $T\in A\ot B\ot C$ equals the
rank of the linear subspace $T(A^*)\subset B\ot C$, see, e.g., \cite[Thm. 3.1.1.1]{Ltensorbook}.

\begin{proposition}  \label{lastprop}
The  ranks of the spaces parametrized by
$
\begin{pmatrix} u&t&s \\ t&s&0 \\ s & 0 & 0\end{pmatrix}$,
 and by
$
\begin{pmatrix}
  t & s &u\\
  s & 0 & 0\\
  u & 0 & 0
\end{pmatrix}
$
are both $5$.
\end{proposition}
\begin{proof}
We first show the rank is at most $5$: in the second case, it is immediate. In the first case the rank of
$
\begin{pmatrix}
  0&t&s\\
  t&s&0\\
  s&0&0
\end{pmatrix}$
is $4$ (see\cite[\S 6]{BLtensor}), and the rank of
$$
 \begin{pmatrix} u &0&0\\ 0&0&0\\  0&0&0\end{pmatrix}
$$
is one.

To see the ranks are at least five,
  were it four in the first case,  we would be able to find
  a $3 \times 3$ matrix $T=
  \begin{pmatrix}
     f_1g_1& f_1g_2 & f_1g_3\\
     f_2g_1& f_2g_2 & f_2g_3\\
     f_3g_1& f_3g_2 & f_3g_3
  \end{pmatrix}$
  of rank 1,
  such that  the $4$-plane spanned by:
\[
T_1:=\begin{pmatrix}
0&0&1\\
0&1&0\\
1&0&0
\end{pmatrix},
T_2:=\begin{pmatrix}
0&1&0\\
1&0&0\\
0&0&0
\end{pmatrix},
T_3:=\begin{pmatrix}
1&0&0\\
0&0&0\\
0&0&0
\end{pmatrix},
T
\]
is spanned by matrices of rank $1$.
In particular,
  $T_1$ would be in the span of $T_2, T_3, T$ and another matrix of rank $1$.
Thus we would be able to find constants
  $\b,\gamma,f_1,f_2,f_3,g_1,g_2,g_3$,
  such that the rank of
$$
\begin{pmatrix}
\gamma& \b & 1\\
\b&1&0\\
1&0&0
\end{pmatrix}
+
\begin{pmatrix}
f_1g_1& f_1g_2 & f_1g_3\\
f_2g_1& f_2g_2 & f_2g_3\\
f_3g_1& f_3g_2 & f_3g_3
\end{pmatrix}
$$
is  one.
There are two cases: if $g_3\neq 0$, then we can subtract
$ \frac{g_1}{g_3}$ times the third column from the first, and
$ \frac{g_2}{g_3}$ times the third column from the second  to obtain
$$
\begin{pmatrix}
* & * & 1+f_1g_3\\
* & 1 & f_2g_3\\
1 & 0 & f_2g_3
\end{pmatrix}
$$
which has rank at least two.
If $g_3=0$ the matrix already visibly has rank at least two.
Thus it is impossible to find such constants $\b, \gamma ,f_i,g_i$
and the rank in question is necessarily at least $5$.

The second case is more delicate.
Write all $2 \times 2$ minors of
\[
\begin{pmatrix}
t & s &u\\
s & 0 & 0\\
u & 0 & 0
\end{pmatrix} + x
\begin{pmatrix}                                                                                                                                             f_1g_1& f_1g_2 & f_1g_3\\                                                                                                                                   f_2g_1& f_2g_2 & f_2g_3\\                                                                                                                                   f_3g_1& f_3g_2 & f_3g_3                                                                                                                                     \end{pmatrix}
\]
and consider $f_i$ and $g_j$ as parameters of degree $0$, and remaining variables $\a_1 , \a_2, \a_3, x$ of degree $1$.
We claim $(s f_3 - u f_2)^2$ and  $(s g_3 - u g_2)^2$ are in the ideal $\ccI$ generated by minors.
This can be verified by patient calculation, or using a computer algebra system, such as Magma \cite{magma}.
Thus $f_2=f_3=g_2=g_3 =0$, for otherwise we have a degree $1$ equation in the radical ideal $\sqrt{I}$,
  and then the rank $1$ matrices do not span the four dimensional linear space.
But in such a case ${u}^2$ and ${s}^2$  are among the minors, giving $u$ and $s$ as linear equations in $\sqrt{I}$,
  a contradiction.
\end{proof}

\subsection{Singularities}\label{s3smooth}
In this subsection we prove Theorems~\ref{txsmooth} and  \ref{s3sing}.
The strategy is uniform to most cases:
  using the desingularization $\BP \cE \to Sub_{i,j,k}$  as in the beginning paragraph of \S\ref{sec_ranks_and_orbits},
  which is birational away from the locus
  $Sub_{i-1,j,k} \cup Sub_{i,j-1,k} \cup Sub_{i,j,k-1}$,
  we reduce statements to properties of secant varieties of low dimensional Segre products.

\begin{proof}[Proof of Theorem~\ref{txsmooth}]
First note that $\s_2(Seg(\BP A\times \BP B\times \BP C))=Sub_{2,2,2}$. In
particular, any point of $\s_2(Seg(\pp 1\times \pp 1\times \pp 1))=\pp 7$ is a smooth point.
Now just observe that $[a_1\ot b_1\ot c_2+ a_1\ot b_2\ot c_1 + a_2\ot b_1\ot c_1]$ is a smooth point of
$Sub_{2,2,2}$, because it is not contained in $Sub_{2,2,1} \cup Sub_{2,1,2} \cup Sub_{1,2,2}$.
\end{proof}

Similarly, in Theorem~\ref{s3sing} if $\dim A = 2$, then $\s_3(Seg(\BP A\times \BP B\times \BP C))=Sub_{2,3,3}$.
A general point of each type \ref{item_main_thm_normal_form_honest_secant}--\ref{item_main_thm_normal_form_two_tangents}
  is not contained in any of the smaller subspace varieties, so the same argument works.
So we will assume $ \dim A, \dim B, \dim C  \ge 3$.

\begin{lemma}
   Suppose $ \dim A =  \dim B = \dim C  = 3$.
   Then a general point of each component of points of type \ref{item_main_thm_normal_form_two_tangents}
      is a smooth point of $\s_3(Seg(\BP A\times \BP B\times \BP C))$.
\end{lemma}

\begin{proof}
The only defining equations of $\sigma_3(Seg(\pp 2\times\pp 2\times \pp 2))$
  are the $27$ (degree four) Strassen equations.
If we write
  $T=a_1\ot X+ a_2\ot Y+a_3\ot Z$, then $9$ of the equations are the entries of the $3\times 3$ matrix
\be\label{strpoly}
P(T)^s_t=
\sum_{j,k}(-1)^{j+k}(\tdet X^{\hat j}_{\hat k})
(Y^j_tZ^s_k-Y^s_kZ^j_t)
\ene
where $X^{\hat j}_{\hat k}$ is $X$ with its $j$-th row and
$k$-th column removed.
The remaining equations come from permuting the roles of $X, Y, Z$, see, e.g. \cite{LWsecseg}.
Take
$T=a_1\ot (b_1\ot c_2 +  b_2\ot c_1 + b_3 \ot c_3) +a_2\ot b_1\ot c_1 + a_3\ot b_3\ot c_1 $
   as in Table~\ref{table_orbits_general} row $35$.
Writing $T=
a_1\ot X+a_2
\ot Y+a_3\ot Z$, we have
$$
X=\begin{pmatrix} 0&1&0\\  1&0&0\\ 0&0&1\end{pmatrix}, \ \
Y=\begin{pmatrix} 1&0&0\\  0&0&0\\ 0&0&0\end{pmatrix}, \ \
Z=\begin{pmatrix} 0&0&0\\ 0&0&0\\ 1&0&0\end{pmatrix}.
$$
Then
$$
dP_T=
\begin{pmatrix}
     -dx_{2,3}+ dy_{1,3}- dz_{1,2} +dz_{2,1} &       -dz_{2,2} &         dz_{2,3}\\
                           dy_{2,3}-dz_{2,2} &               0 &                0\\
      dx_{2,2}- dy_{2,1}+ dy_{3,3}- dz_{3,2} &       -dy_{2,2} &        -dy_{2,3}\\
\end{pmatrix}
$$
which indeed has six linearly independent differentials.

To argue for the other components, i.e., when $T$ is of the form $34$ or $36$ in Table~\ref{table_orbits_general},
   one can permute the factors $A$, $B$, and $C$.
\end{proof}

\begin{proof}[Proof of Theorem~\ref{s3sing}]
Assume $ \dim A, \dim B, \dim C  \ge 3$.
Since the map $\PP(\cE) \to Sub_{3,3,3}$ is an isomorphism
  near a general point of type \ref{item_main_thm_normal_form_two_tangents},
  the Lemma implies that such a point is a smooth point of
  $\s_3(Seg(\BP A\times \BP B\times \BP C))$ for any $A$, $B$, $C$  (each of dimension at least $3$).
But orbits $34$--$36$ from Table~\ref{table_orbits_general} are in the closure of     orbits $37$  and $38$.
So $\s_3(Seg(\BP A\times \BP B\times \BP C))$ is non-singular at a general point of each type \ref{item_main_thm_normal_form_point_plus_tangent}--\ref{item_main_thm_normal_form_two_tangents}.

The final thing to prove is that  $\s_3(Seg(\BP A\times \BP B\times \BP C))$
   is non-singular at a general point of $Sub_{233}$.
Let $p$ be such a point. Since  $\s_3(Seg(\BP A\times \BP B\times \BP C))\subset  Sub_{333}$, we may assume $\tdim A=\tdim B=\tdim C=3$.
First note  that $Sub_{233}$ is not contained in $\ptstypeii$, as they  are both irreducible, have the same dimension
and $\ptstypeii\not\subset Sub_{233}$.
So   $p$ is not in $\ptstypeii$.
By Theorem~\ref{s3nformthm}, this implies that there  exists   an open neighborhood
   $U \subset \sigma_3(\BP A\times \BP B\times \BP C)$ of $p$,
   such that in this neighborhood all points are of type~\ref{item_main_thm_normal_form_honest_secant}.

Consider the dominant rational map
\begin{align*}
   \phi: (A \times B \times C)^{\times 3} & \dashrightarrow \hat \sigma_3(Seg(\PP A \times \PP B \times \PP C)) \\
   (a_1,b_1, c_1),(a_2,b_2, c_2),(a_3,b_3, c_3) &
                   \mapsto a_1 \otimes b_1 \otimes c_1 +a_2 \otimes b_2 \otimes c_2 + a_3 \otimes b_3 \otimes c_3
\end{align*}
Let $W := \phi^{-1} (U)$.
Then $\phi|_{W}: W \to U$ is a regular surjective map.
The aim is to calculate the tangent map at any point in $\phi^{-1}(p)$.
We commence with identifying $\phi^{-1}(p)$.
Since $R_X(p) =3$, any point in $\phi^{-1} (p)$ will be contained in a fixed $(A' \times B' \times C')^{\times 3}$
  with $\dim A' =2$, $\dim B' =\dim C' = 3$ by \cite[Cor.~2.2]{BLtensor}.

Write $p = [a_1 \otimes b_1 \otimes c_1  + (a_1 +a_2) \otimes b_2 \otimes c_2 + a_2 \otimes b_3 \otimes c_3]$
(see \cite[\S 6]{BLtensor}).
We claim that this normal form is  unique  up to trivialities  such as $7$-dimensions worth of rescalings,
   and permutations of summands.
By writing $p:(A')^* \to B' \otimes C'$,
   we obtain the slice
$\left(\begin{smallmatrix}
   s &     &   \\
     & s+t &   \\
     &     & t
\end{smallmatrix} \right)$.
The set of rank $2$ elements in this linear space is given by the determinant of the matrix.
This set consists of three lines in $(A')^*$ spanned by $a_1^*$, $a_1^* - a^*_2$, and $a_2^*$.
Thus the triple $a_1, (a_1 + a_2), a_2$ is (up to order and scale) determined by $p$.
In a similar way we consider the other slices, and $2\times 2$ minors of the resulting matrices, to conclude,
  that triples $b_1, b_2, b_3$ and $c_1, c_2, c_3$ are determined by $p$, up to order and scale.
It is easy to see, that any meaningfully  different choice of orders, or scaling will give   a different     tensor,
  so the preimage of $p$ consists of $6$ components, each of dimension $7$, isomorphic to $(\CC^*)^7$.

\begin{table}[tb]
$$
\begin{array}{|c|r|r|}
\hline
\hline
A\otimes B \otimes C & \dim \sigma_3 & \dim Sing \le   \\
\hline
\hline
\CC^2\otimes \CC^2 \otimes \CC^2                &       7 &      -1  \\
\hline
\CC^2\otimes \CC^2 \otimes \CC^3                &      11 &      -1 \\
\hline
\CC^2\otimes \CC^2 \otimes \CC^{\ccc}           & 3\ccc+2 & 2\ccc+3  \\
\hline
\CC^2\otimes \CC^3 \otimes \CC^{3}              &      17 &     -1   \\
\hline
\CC^2\otimes \CC^3 \otimes \CC^{\ccc}           & 3\ccc+8 & 3\ccc+4  \\
\hline
\CC^2\otimes \CC^{\bbb} \otimes \CC^{\ccc}      & 3\bbb+3\ccc-1 &
                     \begin{matrix} \max \{ 3\bbb+3\ccc-9,\\ 2\bbb+3\ccc-2\} \end{matrix}  \\
\hline
\CC^3\otimes \CC^3 \otimes \CC^{3}              &      20 &     18   \\
\hline
\CC^3\otimes \CC^3 \otimes \CC^{\ccc}           & 3\ccc+11& 3\ccc+9  \\
\hline
\CC^3\otimes \CC^{\bbb} \otimes \CC^{\ccc}      & 3\bbb+3\ccc+2& 3\bbb+3\ccc   \\
\hline
\CC^{\aaa}\otimes \CC^{\bbb} \otimes \CC^{\ccc} & 3\aaa+3\bbb+3\ccc-7 & 2\aaa+3\bbb+3\ccc-6  \\
\hline
\hline
\end{array}
$$
\caption{Singularities of  $\s_3(Seg(\BP A\times \BP B\times \BP C))$.
         In the first column we list the tensor space, assuming $4 \le \aaa \le \bbb \le \ccc$.
         In the second column we write the dimension of the secant variety.
         In the third column we present the upper bound on the dimension of the singular locus of the secant variety,
            which follows from our results in this section.}
\label{table_sings}
\end{table}

Next, we calculate the image of tangent map of $\phi$ at any $q \in \phi^{-1} (p)$,
   say $q= [(a_1,b_1,c_1), (a_1 +a_2,b_2,c_2), (a_2,b_3,c_3)$.
This image is spanned by the following tensors, all considered modulo $p$,
       as we look at a subspace of $T_p \PP(A\otimes B \otimes C) \simeq (A\otimes B \otimes C)/p$:
\begin{align*}
   a_i\otimes b_1 \otimes c_1 && a_i\otimes b_2 \otimes c_2 && a_i\otimes b_3 \otimes c_3 &&
                 \text{for any } i \in\setfromto{1}{\dim A}, \\
   a_1\otimes b_j \otimes c_1 && (a_1+a_2)\otimes b_j \otimes c_2 && a_2\otimes b_j \otimes c_3 &&
                 \text{for any } j \in\setfromto{1}{\dim B}, \\
   a_1\otimes b_1 \otimes c_k && (a_1+a_2)\otimes b_2 \otimes c_k && a_2\otimes b_3 \otimes c_k &&
                 \text{for any } k \in\setfromto{1}{\dim C}.
\end{align*}
This space is independent of the choice of the order or scalings in $q$.
Also the linear space above has dimension
\[
  3 (\dim A + \dim B + \dim C) -7 = \dim \sigma_3(\BP A\times \BP B\times \BP C),
\]
  because there are $3 (\dim A + \dim B + \dim C)$ tensors listed above,
  and each $a_1 \otimes b_1 \otimes c_1$,  $(a_1 +a_2) \otimes b_2 \otimes c_2$, $a_2 \otimes b_3 \otimes c_3$
  is listed three times and $p$ is a sum of those three tensors.
One can check there no other linear dependencies.

Thus $\phi: W \to \PP (A\otimes B \otimes C)$ is a map with constant rank on an open subset containing $\phi^{-1}(p)$.
Therefore the image is non-singular at $p$ as claimed.
\end{proof}

We summarize our results in Table~\ref{table_sings}. In particular, it follows that $\s_3(Seg(\BP A\times \BP B\times \BP C))$ is always non-singular in codimension $1$, that is, codimension of the singular locus is at least $2$.
Moreover, it is of codimension $2$ if and only if, one of the factors is $\CC^3$, and the others have dimension at least $3$.

\bibliographystyle{amsplain}

\bibliography{BLsecbndry}

\end{document}